\documentclass[pdflatex,sn-mathphys-num]{sn-jnl}

\usepackage{amsmath,amssymb,amsthm,mathtools,bm}
\usepackage{comment}
\usepackage{hyperref}

\usepackage{bbm}

\usepackage{graphicx}%
\usepackage{subcaption}
\usepackage{multirow}%
\usepackage{mathrsfs}%
\usepackage[title]{appendix}%
\usepackage{xcolor}%
\usepackage{textcomp}%
\usepackage{manyfoot}%
\usepackage{booktabs}%
\usepackage{algorithm}%
\usepackage{algorithmicx}%
\usepackage{algpseudocode}%
\usepackage{listings}%
\newcommand{\E}{\mathbb{E}}
\newcommand{\Var}{\mathrm{Var}}
\newcommand{\ip}[2]{\langle #1,#2\rangle}
\newcommand{\R}{\mathbb{R}}

\theoremstyle{thmstyleone}%
\newtheorem{theorem}{Theorem}
\newtheorem{proposition}[theorem]{Proposition}%

\theoremstyle{thmstyletwo}%
\newtheorem{example}{Example}%
\newtheorem{remark}{Remark}%

\theoremstyle{thmstylethree}%
\newtheorem{definition}{Definition}%

\newtheorem{lemma}[theorem]{Lemma}

\newtheorem{assumption}[theorem]{Assumption}

\begin{document}

\title[Article Title]{Refining Cramér--Rao Bound With Multivariate Parameters: An Extrinsic Geometry Perspective}


\author*[]{\fnm{Sunder Ram} \sur{Krishnan}}\email{eeksunderram@gmail.com}



\affil[]{\orgdiv{Department of Computer Science and Engineering}, \orgname{Amrita Vishwa Vidyapeetham}, \orgaddress{\street{Amritapuri}, \city{Kollam}, \postcode{690525}, \state{Kerala}, \country{India}}}




\abstract{We derive a vector generalization of the curvature-corrected Cram\'er--Rao bound (CRB) in the nonasymptotic regime using a Hilbert space square-root embedding. Building on previous scalar results, we establish a \emph{directional} curvature correction derived from the second fundamental form of the model manifold. To obtain matrix-valued refinements, we formulate sufficient conditions for a conservative matrix-level correction using a semidefinite program (SDP) based on sum-of-squares (SOS) relaxations. The framework is rigorously illustrated with two distinct geometries: (i) a curved Gaussian location model, which reveals a characteristic \textit{pinching effect} where directional bounds vanish along principal axes despite non-zero extrinsic curvature and classical subspace-based bounds using the second-order Bhattacharyya matrix provide overly optimistic variance predictions that fail to account for the manifold's directional topology, and (ii) a spherical multinomial model where the curvature is isotropic. Our results demonstrate that while classical second-order corrections using the Bhattacharyya matrix provide useful benchmarks derived from the local coordinate basis, the proposed directional and SOS-certified bounds offer a more faithful and geometry-consistent representation of the directional sensitivity and fundamental limits of estimation in curved statistical families.}




\maketitle

\section{Introduction}

In the multivariate setting, the Cramér–Rao bound (CRB) is naturally matrix-valued, and several refinements have been proposed. Regarding relevant work, we have previously surveyed several key research threads in the context of our work on scalar parameters (see \cite{srk}): (1) Bhattacharyya-type extensions of the CRB (for example, \cite{ghosh1987convergence}), (2) Amari et al.'s \cite{amari} classical work in information geometry and related higher-order asymptotic theory \cite{takeuchi2003second,okamoto1991asymptotic}, (3) intrinsic CRBs on Riemannian manifolds \cite{smith_intrinsic_crb,boumal2013_intrinsic_crb,bouchard2024}, (4) divergence-based relaxations \cite{mishra2020,ashok2020,mishra2021}, and (5) Wasserstein extensions \cite{Li,nishimori2025, garciatrillos_wcrb, amari_matsuda_Wasserstein, fukushi2024flatness, zuo2025} aimed at developing estimation theory, efficiency notions, and algorithmic methods, complementing the classical Fisher--Rao framework. For a complete review, we urge the reader to refer to our earlier work \cite{srk}. 

As noted there, an important work of direct relevance to ours is that of Efron \cite{efron1975}, who introduced the notion of ``statistical curvature'' of a one-parameter family at a given point and showed that large curvature corresponds to a breakdown of the nice properties it enjoys in common with the exponential family of distributions. Specifically, Efron considered the problem of estimating the scalar parameter in a curved exponential family, on the basis of an i.i.d. sample of size \(n\), using a squared error loss function to evaluate consistent and efficient estimators that are smooth functions of the sufficient statistic. It was proved that the \(O(\tfrac{1}{n^2})\) term in the variance \textit{asymptotics} precisely involves the squared statistical curvature, linking it to Rao's influential second-order efficiency theory. 

It is pertinent to note that Efron's curvature is the curvature of an embedding associated with the particular way a parameter space is placed inside a natural, higher-dimensional parameter space, as measured by the second fundamental form. 
A related fact may be found in Theorem 4.4 of \cite{amari}, where the statistical curvature term is identified with the e-curvature of the model manifold. However, to our knowledge, no prior work has developed \emph{nonasymptotic} extrinsic curvature corrections for classical variance bounds with vector parameters providing a geometric refinement that remains valid even outside the large-sample limit. This paper extends our earlier scalar-parameter results to the multivariate setting where more interesting directions are available for investigation, providing a geometric refinement of the CRB.

We first set up the framework as a natural extension of the univariate case. On the measurable space $(\mathcal X,\mathcal F)$, let $\mu$ be a fixed $\sigma$-finite measure. Consider a parametric family of probability measures $\{P_\theta:\theta\in\Theta\subset\mathbb{R}^d\}$ that are absolutely continuous with respect to \(\mu\) with strictly positive densities $f(x;\theta)$. Denote the ambient Hilbert space \(H := L^2(\mu)\).  The statistical model is mapped into \(H\) via the square-root embedding:
\[
s: \Theta \to H, \qquad s(\theta)(x) = \sqrt{f(x;\theta)}.
\]
Note that the image $s(\Theta)$ is a $d$-dimensional submanifold lying on the unit sphere of $H$.
Our first result uses the second fundamental form of \(s(\Theta)\subset H\) to derive a \emph{directional} curvature-aware refinement to the CRB; see Theorem \ref{thm:dir} for a lower bound on \(v^\top (\Sigma-J^{-1})v\), where \(v\in\mathbb{R}^d\), \(\Sigma\) is the covariance matrix of an unbiased estimator of \(\theta\in\Theta\), and \(J\) is the Fisher information matrix (FIM). Specifically, we treat the submanifold $s(\Theta)$ as an embedding into the Hilbert space $H$, where the model-specific extrinsic curvature captures the component of the path-acceleration that is orthogonal to the tangent space $\mathcal{T}_1:=\mathrm{span}\{\partial_is(\theta)\}_{i=1}^d$, characterizing the bending of the model within the ambient geometry of square-root densities.
This directional bound $\mathcal{R}(v)$ reveals a subtle \textit{pinching effect} in multivariate models, where curvature sensitivity may vanish along specific parameter axes despite non-zero extrinsic curvature; that is, the directional curvature may not capture unexplained variance of an unbiased estimator, implying absence of curvature-induced corrections to classical statistical bounds/inefficiency of the estimator along such directions. Following an analysis of the obstacles to achieving a generic matrix-valued correction \(\Sigma-J^{-1}\succeq \Delta\), we conclude that a universal refinement is only possible through a conservative approach. Thus, a certified matrix correction $\Delta$ is derived using a constructive method that employs a semidefinite program (SDP) based on sum-of-squares (SOS) relaxations. This result is presented in Theorem \ref{thm:sdp}. Taking into account the higher-order extrinsic geometry of the embedding $s(\Theta)\subset H$ by considering the corresponding jet spaces, our second-order analysis is extended to result in directional curvature-based corrections to the Bhattacharyya-type higher-order versions of the CRB (cf. Theorem \ref{thm:hor}), allowing for a nested sequence of increasingly tight lower bounds. The paper concludes with two detailed analytical examples in Section \ref{sec:ex}: a curved Gaussian location model illustrating the degenerate pinching effect, and a spherical multinomial model where isotropic geometry allows the SOS-certified bound to recover the classical second-order limits. These examples demonstrate that while classical benchmarks such as the Bhattacharyya matrix are direction-agnostic, our approach explicitly accounts for the directional sensitivity of the manifold's extrinsic curvature, correctly identifying when the `information gap' vanishes along specific axes.

\section{Setup, notation, and assumptions}\label{sec:assumptions}

As stated in the introduction, let $(\mathcal X,\mathcal F)$ be a measurable space with $\mu$ being a fixed $\sigma$-finite measure. We consider a parametric family of probability measures $\{P_\theta:\theta\in\Theta\subset\mathbb{R}^d\}$ possessing strictly positive densities $f(x;\theta) = \frac{dP_\theta}{d\mu}(x)$. Expectations under $P_\theta$ are denoted by $\mathbb{E}_\theta[\cdot]$.

Fix the ambient Hilbert space \(H = L^2(\mu)\). For any functions $u, v \in H$, the inner product and norm are defined as:
\[
\ip{u}{v} := \int_{\mathcal{X}} u(x)v(x)\,d\mu(x), \qquad \|u\|^2 = \ip{u}{u}.
\]
The orthogonal projection of an element $h \in H$ onto a closed subspace \(V \subseteq H\) is denoted by $\mathrm{Proj}_V(h)$. We map the statistical model into the Hilbert space via the square-root embedding:
\[
s: \Theta \to H, \qquad s(\theta)(x) = \sqrt{f(x;\theta)}.
\]
Since $\|s(\theta)\|^2 = \int f(x;\theta) d\mu = 1$, the image of the parameter space $\Theta$ under $s$ is a $d$-dimensional submanifold $\mathcal{M} = s(\Theta)$ lying on the unit sphere of $H$.

For a multi-index $\alpha = (\alpha_1, \dots, \alpha_d) \in \mathbb{N}_0^d$ of order $|\alpha| = \sum \alpha_i$, we denote the partial derivative operator by $\partial^\alpha := \partial_1^{\alpha_1} \dots \partial_d^{\alpha_d}$. We define the $k$-th order score functions $Y_\alpha(x;\theta)$ and the embedded derivative vectors $\eta_\alpha(\theta) \in H$ as:
\[
Y_\alpha(x;\theta) := \frac{\partial^\alpha f(x;\theta)}{f(x;\theta)}, \qquad \eta_\alpha(\theta) := \partial^\alpha s(\theta).
\]
\begin{lemma}[Fa\`a di Bruno Expansion and Statistical Scores]
\label{lem:faa_di_bruno}
The $m$-th order embedded derivative vectors $\eta_\alpha$ can be expressed as a product of the square-root density $s$ and a polynomial in the lower-order score functions $Y_\beta$. Specifically, for orders $1$ and $2$:
\begin{align}
\eta_i &= \frac{1}{2} s Y_i, \label{eq:eta1}\\
\eta_{ij} &= \frac{1}{2} s \left( Y_{ij} - \frac{1}{2} Y_i Y_j \right). \label{eq:eta2}
\end{align}
Generally, for any multi-index $\alpha$, $\eta_\alpha = s \cdot \mathcal{P}_\alpha(Y_{\beta \le \alpha})$, where $\mathcal{P}_\alpha$ is the polynomial determined by the Fa\`a di Bruno formula for the composite function $\sqrt{f(\theta)}$.
\end{lemma}

\begin{proof}
Equation \eqref{eq:eta1} follows from the chain rule $\partial_i \sqrt{f} = \frac{\partial_i f}{2\sqrt{f}}$. Differentiating again with respect to $\theta_j$ yields:
\[
\partial_j \eta_i = \partial_j \left( \frac{\partial_i f}{2\sqrt{f}} \right) = \frac{\partial_{ij} f}{2\sqrt{f}} - \frac{\partial_i f \partial_j f}{4 f^{3/2}} = \frac{1}{2} \sqrt{f} \left( \frac{\partial_{ij} f}{f} - \frac{1}{2} \frac{\partial_i f}{f} \frac{\partial_j f}{f} \right).
\]
Substituting $s = \sqrt{f}$ and $Y_\alpha = \partial^\alpha f / f$ gives \eqref{eq:eta2}.
\end{proof}

This identifies the tangent space $\mathcal{T}_1(\theta) = \mathrm{span}\{\eta_i\}_{i=1}^d$ as the span of the classical score functions in $H$. 

We define the $m$-th order jet space at $\theta$ as the subspace:
\[
\mathcal{T}_m(\theta) := \mathrm{span} \left\{ \eta_\alpha(\theta) : 1 \le |\alpha| \le m \right\} \subset H.
\]
The total dimension of this space is $D_m = \binom{d+m}{m}-1$.

\begin{remark}[Statistical Significance of the Expansion]
The Fa\`a di Bruno expansion is critical for the operational use of the curvature bounds derived in our work. First, it allows for the calculation of the Gramian matrices and other relevant inner products in our further developments as standard expectations under $P_\theta$. For example, the inner products in $H$ translate to:
\[
\ip{\eta_{ij}}{\eta_{kl}} = \frac{1}{4} \mathbb{E}_\theta \left[ \left( Y_{ij} - \frac{1}{2} Y_i Y_j \right) \left( Y_{kl} - \frac{1}{2} Y_k Y_l \right) \right].
\]
Second, it reveals the algebraic structure of the jet space; to give a taste of what is to come, since $\eta_{ij}$ is a linear combination of the second-order score $Y_{ij}$ and the product of first-order scores, the extrinsic curvature -- the portion of $\eta_{ij}$ orthogonal to the tangent space $\mathcal{T}_1$ -- is governed by the relationship between these statistical moments. 
\end{remark}

Throughout the paper, we assume the following regularity conditions.

\begin{assumption}[Regularity of order $m$]
\label{assump:smooth}
\begin{enumerate}[\rm (i)]

\item \textbf{Differentiability:} The map $s: \Theta \to H$ is $m$-times Fr\'echet differentiable. The derivative vectors $\{\eta_\alpha\}_{1 \le |\alpha| \le m}$ are $\mu$-measurable and lie in $H$ for every $\theta \in \Theta$.
\item \textbf{Leibniz Rule:} For any estimator $T(X)$ such that $\mathbb{E}_\theta[T^2] < \infty$, the following identity holds for all $|\alpha| \le m$:
\[
\partial^\alpha \int T(x) f(x;\theta) d\mu = \int T(x) \partial^\alpha f(x;\theta) d\mu.
\]
\item \textbf{Non-degeneracy:} The $m$-th order Gram matrix $G^{(m)}(\theta)$ with entries $G_{\alpha, \beta} = \ip{\eta_\alpha}{\eta_\beta}$ is strictly positive definite (full rank) at the parameter $\theta$.
\end{enumerate}
\end{assumption}

Under Assumption \ref{assump:smooth}, the FIM $J(\theta)$ is exactly $4G^{(1)}(\theta)\equiv 4G(\theta)$, i.e., $J_{ij}(\theta) = 4\ip{\eta_i}{\eta_j}$. Furthermore, the non-degeneracy condition ensures that the derivatives are linearly independent, so that $\dim(\mathcal{T}_m) = D_m$.

\section{Projection proof of the classical matrix CRB}\label{sec:CRB}

Although not novel, we initially establish a simple result that presents the classical matrix CRB in terms of the projection of the estimator error onto the tangent space of the embedded statistical manifold $\mathcal{M}$. This projection perspective serves as a rigorous springboard for the higher-order geometric developments that follow.

With \(X\sim P_\theta\), let \(T(X)\in\mathbb{R}^d\) be an unbiased estimator: \(\mathbb{E}_\theta[T(X)]=\theta\).
Define the centered error vector \(Z_0 := T(X)-\theta\in\mathbb{R}^d\) and lift its components to the Hilbert space \(H\) by
\[
\widetilde Z^{(p)}(x) := Z_0^{(p)}\, s(\theta)(x)\in H,\qquad p=1,\dots,d.
\]
By construction, the inner product of these lifted components recovers the covariance:
\[
\langle\widetilde Z^{(p)},\widetilde Z^{(q)}\rangle = \int_{\mathcal{X}} (T_p(x)-\theta_p)(T_q(x)-\theta_q) f(x;\theta) \, d\mu = \mathbb{E}_\theta[Z_0^{(p)} Z_0^{(q)}] =: \Sigma_{pq}(\theta).
\]
Recall the tangent vectors \(\eta_i(\theta)=\partial_i s(\theta)\in H\) and the tangent subspace \(\mathcal{T}_1 = \operatorname{span}\{\eta_1,\dots,\eta_d\} \subset H\).

\begin{proposition}[Projection inequality]\label{prop:proj_ineq}
Let \(B\in\mathbb{R}^{d\times d}\) be the matrix with entries defined by the inner products of the projected lifted errors:
\[
B_{pq} := \big\langle \operatorname{Proj}_{\mathcal{T}_1}\widetilde Z^{(p)},\, \operatorname{Proj}_{\mathcal{T}_1}\widetilde Z^{(q)}\big\rangle,
\]
where \(\operatorname{Proj}_{\mathcal{T}_1}:H\to \mathcal{T}_1\) is the orthogonal projection. Then the estimator covariance satisfies:
\[
\Sigma(\theta) \succeq B.
\]
\end{proposition}

\begin{proof}
Fix an arbitrary direction vector \(v=(v_1,\dots,v_d)^\top\in\mathbb{R}^d\). Define the combined lifted error
\[
Z_v := \sum_{p=1}^d v_p\,\widetilde Z^{(p)} \in H.
\]
The squared norm of this vector in \(H\) corresponds to the variance in direction \(v\):
\[
\| Z_v \|^2 = \sum_{p,q} v_p v_q \langle \widetilde Z^{(p)}, \widetilde Z^{(q)} \rangle = v^\top\Sigma(\theta)v.
\]

Since the orthogonal projection onto a closed subspace of a Hilbert space is a contraction, we have:
\[
\|Z_v\|^2 \ge \big\| \operatorname{Proj}_{\mathcal{T}_1} Z_v\big\|^2.
\]
By the linearity of the projection operator,
\[
\operatorname{Proj}_{\mathcal{T}_1} Z_v = \sum_{p=1}^d v_p\,\operatorname{Proj}_{\mathcal{T}_1}\widetilde Z^{(p)}.
\]
The squared norm of the projected vector is then:
\[
\big\| \operatorname{Proj}_{\mathcal{T}_1} Z_v\big\|^2
= \sum_{p,q} v_p v_q\,\big\langle \operatorname{Proj}_{\mathcal{T}_1}\widetilde Z^{(p)},\operatorname{Proj}_{\mathcal{T}_1}\widetilde Z^{(q)}\big\rangle
= v^\top B v.
\]
Combining these results, we obtain \(v^\top\Sigma v \ge v^\top B v\) for all \(v\in\mathbb{R}^d\), which implies the matrix inequality \(\Sigma\succeq B\).
\end{proof}

\begin{proposition}[Identification of the projection matrix: $B=J^{-1}$]\label{prop:B_is_Jinv}
Under the regularity conditions of Assumption \ref{assump:smooth}, the matrix \(B\) defined in Proposition \ref{prop:proj_ineq} is identically the inverse FIM:
\[
B = J(\theta)^{-1}.
\]
\end{proposition}

\begin{proof}
The Gram matrix of the tangent vectors \(G \in \mathbb{R}^{d\times d}\) has elements \(G_{ij} = \langle\eta_i,\eta_j\rangle = \frac{1}{4} J_{ij}\). Since \(G\) is invertible by assumption, the orthogonal projection of any \(u \in H\) onto \(\mathcal{T}_1\) is given by:
\[
\operatorname{Proj}_{\mathcal{T}_1} u = \sum_{j=1}^d \left( G^{-1} \, b(u) \right)_j \,\eta_j,
\qquad \text{where } b(u)_j := \langle u,\eta_j\rangle.
\]
We apply this to \(u = \widetilde Z^{(p)} = Z_0^{(p)} s\). The projection coefficients are:
\[
b^{(p)}_j = \langle Z_0^{(p)} s, \tfrac{1}{2} s Y_j \rangle = \tfrac{1}{2} \mathbb{E}_\theta[Z_0^{(p)} Y_j].
\]
Differentiating the unbiasedness condition \(\mathbb{E}_\theta[T_i(X)] = \theta_i\) with respect to \(\theta_j\) and applying the Leibniz rule from Assumption \ref{assump:smooth} yields:
\[
\delta_{ij} = \frac{\partial}{\partial \theta_j} \int T_i(x) f(x;\theta) d\mu = \int T_i(x) \partial_j f(x;\theta) d\mu = \mathbb{E}_\theta[T_i Y_j].
\]
Since \(\mathbb{E}_\theta[Y_j] = 0\), we have \(\mathbb{E}_\theta[Z_0^{(i)} Y_j] = \mathbb{E}_\theta[(T_i - \theta_i) Y_j] = \delta_{ij}\). Thus, for each \(p\), the coefficient vector is \(b^{(p)} = \frac{1}{2} e_p\), where \(e_p\) is the \(p\)-th standard basis vector. The projected error is:
\[
\operatorname{Proj}_{\mathcal{T}_1}\widetilde Z^{(p)} = \tfrac{1}{2} \sum_{j=1}^d (G^{-1})_{jp} \eta_j.
\]
The entries of matrix \(B\) are then:
\[
B_{pq} = \tfrac{1}{4} \sum_{j,k} (G^{-1})_{jp} (G^{-1})_{kq} \langle \eta_j, \eta_k \rangle = \tfrac{1}{4} (G^{-1} G G^{-1})_{pq} = \tfrac{1}{4} (G^{-1})_{pq}.
\]
Given \(G = \frac{1}{4} J\), it follows that \(G^{-1} = 4 J^{-1}\), hence \(B_{pq} = (J^{-1})_{pq}\).
\end{proof}

The projection proof of the classical CRB demonstrates that the ``efficiency gap" \(\Sigma - J^{-1}\) is precisely the squared norm of the portion of the lifted error \(Z_v\) lying in the orthogonal complement of the tangent space. We now move beyond the tangent space to analyze this residual using the extrinsic geometry of the manifold \(s(\Theta)\).

\section{Induced connection and the second fundamental form}\label{sec:geometry}

Consider the flat ambient derivative in the Hilbert space $H$. The geometry of the embedded manifold $\mathcal{M}$ is characterized by how these ambient derivatives decompose into components tangent and normal to the manifold. Specifically, the induced connection on the tangent bundle of \(\mathcal{M}\subset H\) is defined by projecting these ambient derivatives back onto the tangent space \(\mathcal{T}_1\).

\begin{definition}[Induced connection]
For the coordinate vector fields $\partial_i$ and the corresponding tangent vectors $\eta_j = \partial_j s$, the induced connection $\nabla^{\mathrm{ind}}$ is defined as:
\[
\nabla^{\mathrm{ind}}_{\partial_i}\eta_j := \mathrm{Proj}_{\mathcal{T}_1}\big(\partial_i\eta_j\big)
= \sum_{\ell=1}^d \Gamma^\ell_{ij}\,\eta_\ell,
\]
where the Christoffel symbols of the second kind are given by:
\[
\Gamma^\ell_{ij}(\theta) = \sum_{m=1}^d \ip{\partial_i\eta_j}{\eta_m}\,(G^{-1})_{m\ell}.
\]
\end{definition}

The portion of the second derivative that is lost during this projection represents the extrinsic ``bending" of the manifold within the Hilbert space. This is captured by the second fundamental form.

\begin{definition}[Second fundamental form]
The second fundamental form is the normal-valued symmetric bilinear form $\Pi: \mathcal{T}_1 \times \mathcal{T}_1 \to \mathcal{T}_1^\perp$ defined by the Gauss formula:
\[
\Pi(\partial_i,\partial_j) := \partial_i\eta_j - \nabla^{\mathrm{ind}}_{\partial_i}\eta_j = \partial_i\eta_j - \sum_{\ell=1}^d \Gamma^\ell_{ij}\,\eta_\ell.
\]
We shall denote the components of this form as $\Pi_{ij} := \Pi(\partial_i, \partial_j)$. By construction, $\Pi_{ij} \in \mathcal{T}_1^\perp$ for all $i, j \in \{1, \dots, d\}$.
\end{definition}

\begin{proposition}[Symmetry of $\Pi$]
For the square-root embedding $s:\Theta\to H$ and the second fundamental form defined above, the components satisfy the symmetry relation $\Pi_{ij} = \Pi_{ji}$ as elements of the normal space $\mathcal{T}_1^\perp$.
\end{proposition}

\begin{proof}
By definition, the components are given by:
\[
\Pi_{ij} = \partial_i\eta_j - \sum_{m,\ell=1}^d \ip{\partial_i\eta_j}{\eta_m}\,(G^{-1})_{m\ell}\,\eta_\ell.
\]
To establish equality, it suffices to show that their projections onto any vector $Z$ in the normal space $\mathcal{T}_1^\perp$ coincide. Let $Z \in \mathcal{T}_1^\perp$, which implies $\ip{\eta_\ell}{Z} = 0$ for all $\ell \in \{1, \dots, d\}$. Then:
\begin{align*}
\ip{\Pi_{ij}}{Z} &= \ip{\partial_i\eta_j}{Z} - \sum_{m,\ell} \ip{\partial_i\eta_j}{\eta_m} (G^{-1})_{m\ell} \ip{\eta_\ell}{Z} \\
&= \ip{\partial_i\eta_j}{Z}.
\end{align*}
Similarly, we find $\ip{\Pi_{ji}}{Z} = \ip{\partial_j\eta_i}{Z}$. Under the regularity conditions of Assumption \ref{assump:smooth}, the embedding $s$ is at least $C^2$ in the Fr\'echet sense, and the mixed partial derivatives commute:
\[
\partial_i\eta_j = \partial_i\partial_j s = \partial_j\partial_i s = \partial_j\eta_i.
\]
It follows that $\ip{\partial_i\eta_j}{Z} = \ip{\partial_j\eta_i}{Z}$ for all $Z \in \mathcal{T}_1^\perp$, which implies $\Pi_{ij} = \Pi_{ji}$.
\end{proof}

The second fundamental form encodes the extrinsic curvature of the statistical model. In the following section, we demonstrate how this geometric object provides a correction term to refine the directional CRB.

\section{Directional curvature-corrected CRB}

With \(X\sim P_\theta\), let \(T(X)\in\mathbb{R}^d\) be an unbiased estimator of the vector parameter \(\theta\). As before, write the centered error \(Z_0(x)=T(x)-\theta\in\mathbb{R}^d\), and lift componentwise:
\[
\widetilde Z^{(p)}(x) = Z_0^{(p)}(x)\, s(\theta)(x) \in H,\qquad p=1,\dots,d.
\]
For any direction $v\in\mathbb{R}^d$, we define the combined lifted error $Z_v$ and the covariant direction vector $\widetilde v$ as:
\[
Z_v := \sum_{p=1}^d v_p \widetilde Z^{(p)} \in H,\qquad
\widetilde v := G^{-1} v\in\mathbb{R}^d.
\]
The use of the inverse Gram matrix $G^{-1}$ in the definition of \(\widetilde v\) is needed to ensure that the resulting construction is coordinate-invariant, effectively mapping the direction $v$ into the dual space via the metric.

\begin{definition}[Directional curvature vector]
The curvature (normal) vector associated with the direction $v$ is defined by the second fundamental form applied to the covariant vector $\widetilde v$:
\[
\Pi_v := \sum_{i,j=1}^d \widetilde v_i \widetilde v_j \,\Pi_{ij} \in \mathcal{T}_1^\perp.
\]
\end{definition}

Geometrically, $\Pi_v$ represents the specific normal component of the manifold's acceleration when moving in the direction of the tangent vector $\sum \widetilde v_i \eta_i$.

\begin{theorem}[Directional curvature-corrected CRB]\label{thm:dir}
Under the regularity of Assumption \ref{assump:smooth}, for every direction $v\in\mathbb{R}^d$, the estimator covariance $\Sigma(\theta)$ satisfies:
\begin{equation}
\label{eq:di}
v^\top\!\big(\Sigma(\theta)-J(\theta)^{-1}\big) v
= \|Z_v - \mathrm{Proj}_{\mathcal{T}_1} Z_v\|^2
\ge \frac{\ip{Z_v}{\Pi_v}^{2}}{\|\Pi_v\|^{2}},
\end{equation}
where the right-hand side is taken to be $0$ if $\Pi_v=0$. Equivalently, in terms of the component pairings:
\begin{equation}
\label{eq:ex}
v^\top(\Sigma-J^{-1})v \ge
\frac{\Big(\sum_{p,i,j} v_p\,\widetilde v_i\widetilde v_j\,
\ip{\widetilde Z^{(p)}}{\Pi_{ij}}\Big)^2}
{\big\|\sum_{i,j}\widetilde v_i\widetilde v_j\,\Pi_{ij}\big\|^2}.
\end{equation}
\end{theorem}

\begin{proof}
By Proposition \ref{prop:proj_ineq}, the term $v^\top(\Sigma-J^{-1})v$ is precisely the squared norm of the residual error $R_v := Z_v - \mathrm{Proj}_{\mathcal{T}_1} Z_v$. Since the projection of $Z_v$ onto the tangent space $\mathcal{T}_1$ is orthogonal, the residual $R_v$ lies entirely in the normal space $\mathcal{T}_1^\perp$.

We have that $\Pi_v \in \mathcal{T}_1^\perp$. Applying the Cauchy--Schwarz inequality in $H$ to the vectors $R_v$ and $\Pi_v$, we obtain:
\[
\|R_v\|^2 \ge \frac{\ip{R_v}{\Pi_v}^2}{\|\Pi_v\|^2}.
\]

Because $\mathrm{Proj}_{\mathcal{T}_1} Z_v$ is orthogonal to all elements of the normal space, we have $\ip{\mathrm{Proj}_{\mathcal{T}_1} Z_v}{\Pi_v} = 0$. Consequently, $\ip{R_v}{\Pi_v} = \ip{Z_v}{\Pi_v}$, which completes the proof.
\end{proof}

\begin{remark}[Reduction to the scalar case \(d=1\)]
When \(d=1\), Theorem~\ref{thm:dir} recovers the scalar curvature-corrected CRB previously established in \cite{srk}. For a scalar parameter \(\theta\), let \(\eta = \partial_\theta s\), and \(\Pi = \Pi(\partial_\theta, \partial_\theta)\). For any nonzero direction \(v \in \mathbb{R}\), we have:
\[
\widetilde v = G^{-1} v = \|\eta\|^{-2} v, \quad Z_v = v \widetilde Z, \quad \Pi_v = (\|\eta\|^{-2} v)^2 \Pi.
\]
Substituting these into the bound:
\[
v^2 (\Var_\theta[T] - I(\theta)^{-1}) \ge \frac{(v \cdot (\|\eta\|^{-2} v)^2 \ip{\widetilde Z}{\Pi})^2}{(\|\eta\|^{-2} v)^4 \|\Pi\|^2} = v^2 \frac{\ip{\widetilde Z}{\Pi}^2}{\|\Pi\|^2}.
\]
Dividing by \(v^2\) yields the scalar result \(\Var_\theta[T] - I(\theta)^{-1} \ge \frac{\ip{\widetilde Z}{\Pi}^2}{\|\Pi\|^2}\).
\end{remark}

This directional bound provides a local, high-fidelity characterization of the variance gap. However, the bound $\mathcal{R}(v):=\frac{\ip{Z_v}{\Pi_v}^{2}}{\|\Pi_v\|^{2}}$ is a rational function, which limits its ability to be represented by a single positive semidefinite (PSD) matrix and motivates the SDP approach.

\section{Matrix curvature correction}
\label{sec:mc}
In Theorem \ref{thm:dir}, we derived a family of directional inequalities. A natural question arises: Is it possible to compress this infinite set of directional constraints into a single symmetric PSD matrix $\Delta$ such that the matrix inequality $\Sigma \succeq J^{-1} + \Delta$ holds?

\subsection{On ``directional'' vs ``matrix'' curvature corrections}
\label{sec:dm}
Recalling the result \eqref{eq:di} from Theorem \ref{thm:dir}, we define the numerator and denominator polynomials:
\[
N(v) := \langle Z_v,\Pi_v\rangle,\qquad D(v) := \|\Pi_v\|^2.
\]
By construction, \(Z_v\) is linear in the coordinates of \(v\), while \(\Pi_v\) is quadratic in \(v\). It follows that:
\[
N(v)\ \text{is a homogeneous polynomial of degree } 3,\] and \[D(v)\ \text{is a homogeneous polynomial of degree } 4.
\]
The directional curvature correction \(\mathcal R(v)\) appearing in Theorem~\ref{thm:dir} can then be expressed as:
\[
\mathcal R(v) \:=\; \frac{N(v)^2}{D(v)}.
\]
Observe that \(\mathcal R(v)\) is homogeneous of degree \(6-4=2\) in \(v\). Thus, \(\mathcal R\) is a homogeneous degree-2 function on \(\mathbb{R}^d\). However, in general, it is a rational function rather than a polynomial. This distinction is critical: a matrix-valued correction corresponds to a quadratic polynomial \(v^\top \Delta v\), whereas the true geometric limit is a more complex rational surface.

\begin{proposition}[Algebraic obstruction to an exact matrix representation]
There exists a symmetric matrix \(\Delta \in \mathbb{R}^{d\times d}\) such that
\[
v^\top \Delta v = \frac{N(v)^2}{D(v)} \qquad \text{for all } v \in \mathbb{R}^d
\tag{*}
\]
if and only if the degree-6 homogeneous polynomial \(N(v)^2\) is divisible by the degree-4 homogeneous polynomial \(D(v)\), with the quotient being a homogeneous quadratic polynomial. Equivalently, (*) holds if and only if there exists a homogeneous quadratic polynomial \(Q(v)\) such that:
\[
N(v)^2 = Q(v)\,D(v) \qquad (\text{as an identity of polynomials}).
\]
\label{prop:alg}
\end{proposition}

\begin{proof}
\(\Rightarrow\) If such a \(\Delta\) exists, multiplying both sides of (*) by \(D(v)\) yields:
\[
N(v)^2 - (v^\top \Delta v) D(v) \equiv 0.
\]
The left-hand side is a homogeneous polynomial of degree \(6\). For this to vanish identically, \(D(v)\) must be a factor of \(N(v)^2\), and the resulting quotient must be the quadratic form \(v^\top \Delta v\).

\(\Leftarrow\) Conversely, if \(N^2 = Q \cdot D\) for some homogeneous quadratic \(Q(v)\), then since every quadratic form corresponds to a unique symmetric matrix, we can write \(Q(v) = v^\top \Delta v\). This matrix \(\Delta\) then satisfies \((*)\) by construction.
\end{proof}

\begin{remark}
The polynomial identity \(N^2 - (v^\top\Delta v)D \equiv 0\) represents a system of linear equations, one for each of the \(\binom{d+5}{6}\) possible monomials of degree 6. The number of unknowns is the number of independent entries in \(\Delta\), which is \(\frac{d(d+1)}{2}\). For \(d \ge 2\), the number of equations significantly exceeds the number of variables (e.g., for \(d=2\), we have 7 equations for 3 variables). Generically, this system is overdetermined, implying that an exact matrix representation only exists when the second fundamental form \(\{\Pi_{ij}\}\) and error pairings \(\{\widetilde Z^{(p)}\}\) possess specific symmetries.
\end{remark}

\begin{remark}
\label{rem:mat}
A sufficient condition for the existence of an exact matrix correction \(\Delta\) occurs when the normal components of the second fundamental form span a one-dimensional subspace. Specifically, if there exists a vector \(\phi \in \mathcal{T}_1^\perp\) and a homogeneous quadratic scalar \(h(v)\) such that \(\Pi_v = h(v)\phi\) for all \(v\), and the pairings \(\langle \widetilde Z^{(p)}, \phi \rangle = a_p\) are constant, then:
\[
\mathcal{R}(v) = \frac{\left(\sum_p v_p a_p\right)^2 h(v)^2}{h(v)^2 \|\phi\|^2} = v^\top \left(\frac{a a^\top}{\|\phi\|^2}\right) v,
\]
where \(a = (a_1, \dots, a_d)^\top\). In this isotropic case, the exact PSD matrix correction is \(\Delta = \frac{a a^\top}{\|\phi\|^2}\).
\end{remark}

\subsection{A conservative matrix correction} 
As discussed, an exact matrix representation is not generally possible. The directional curvature correction from Theorem~\ref{thm:dir} provides:
\[
v^\top(\Sigma-J^{-1})v \;\ge\; \mathcal R(v) = \frac{\langle Z_v,\Pi_v\rangle^2}{\|\Pi_v\|^2}.
\]
To obtain a matrix-valued bound that is valid across the entire parameter space, we seek a symmetric PSD matrix $\Delta \in \mathbb{R}^{d\times d}$ such that $v^\top(\Sigma-J^{-1})v \ge \mathcal{R}(v)\geq v^\top \Delta v$ for all $v \in \mathbb{R}^d$, so that \(\Delta\) is faithful to the directional measure of estimator inefficiency captured in \(\mathcal{R}(v)\). This is achieved by finding the ``largest" matrix ellipse that remains entirely contained within the directional structure defined by \(\mathcal{R}(v)\).

\subsection*{Indexing and preparatory definitions}
Let $\mathcal I = \{(i,j) : 1 \le i \le j \le d\}$ index the symmetric pairs of indices, with $m := |\mathcal I| = \frac{d(d+1)}{2}$. For each $\alpha \in \mathcal I$, let $\Pi_\alpha$ denote the corresponding normal vector $\Pi_{ij}$. We define the normal Gramian matrix \(G_{\mathcal{N}} \in \mathbb{R}^{m\times m}\) as:
\[
(G_{\mathcal{N}})_{\alpha\beta} = \ip{\Pi_\alpha}{\Pi_\beta}.
\]
For each estimator component $p \in \{1, \dots, d\}$, we define the projection coefficients onto the normal vectors:
\[
c_{p,\alpha} := \ip{\widetilde Z^{(p)}}{\Pi_\alpha}.
\]
These are collected into a matrix $C \in \mathbb{R}^{d\times m}$ where $C_{p,\alpha} = c_{p,\alpha}$. We further define a vector of quadratic monomials $s(v) \in \mathbb{R}^{m}$ and a linear functional $d(v) \in \mathbb{R}^m$ as:
\[
s_\alpha(v) := \widetilde v_i \widetilde v_j, \qquad d(v) := C^\top v.
\]
Using these definitions, the relevant directional quantities are expressed as:
\[
\langle Z_v,\Pi_v\rangle = d(v)^\top s(v), \qquad \|\Pi_v\|^2 = s(v)^\top G_{\mathcal{N}} s(v).
\]
Consequently, the directional lower bound is compactly represented as the ratio
\[
\mathcal R(v) = \frac{(d(v)^\top s(v))^2}{s(v)^\top G_{\mathcal{N}} s(v)}.
\]

\subsection*{Goal} 
We seek a symmetric PSD matrix $\Delta \in \mathbb{R}^{d\times d}$ such that the following inequality holds for all directions $v \in \mathbb{R}^d$:
\begin{equation}\label{eq:goal}
v^\top\Delta v \;\le\; \frac{(d(v)^\top s(v))^2}{s(v)^\top G_{\mathcal{N}} s(v)}.
\end{equation}
Consistent with Theorem~\ref{thm:dir}, we interpret the right-hand side (RHS) as $0$ whenever the denominator $s(v)^\top G_{\mathcal{N}} s(v) = \|\Pi_v\|^2$ vanishes. If a matrix $\Delta$ satisfying \eqref{eq:goal} is found, the directional curvature bound immediately implies $v^\top(\Sigma - J^{-1} - \Delta)v \ge 0$, which yields the matrix inequality $\Sigma \succeq J^{-1} + \Delta$.

\subsection*{Polynomial (SOS) reformulation}
To avoid the complexities of direct rational optimization, we multiply both sides of \eqref{eq:goal} by the non-negative scalar denominator $s(v)^\top G_{\mathcal{N}} s(v)$. This yields the homogeneous polynomial inequality:
\[
P_\Delta(v) \;:=\; (d(v)^\top s(v))^2 - (v^\top\Delta v)\cdot (s(v)^\top G_{\mathcal{N}} s(v)) \;\ge\; 0, \quad \forall v \in \mathbb{R}^d.
\]
As we already saw in Section \ref{sec:dm}, $P_\Delta$ is a homogeneous polynomial of degree $6$. While checking the non-negativity of a general multivariate polynomial is NP-hard, a sufficient condition is that $P_\Delta(v)$ admits a SOS decomposition. The SOS condition is equivalent to the existence of a symmetric PSD matrix $S \succeq 0$ (the Gram matrix of the SOS form) such that:
\[
P_\Delta(v) = z(v)^\top S z(v),
\]
where $z(v)$ is the vector of all monomials in $v_1, \dots, v_d$ of degree exactly $3$. The length of this basis vector is $M = \binom{d+3-1}{3} = \binom{d+2}{3}$. 

Expanding both sides of the identity $P_\Delta(v) = z(v)^\top S z(v)$ allows us to equate the coefficients of corresponding monomials. Since $P_\Delta$ is linear in the entries of $\Delta$ and the RHS is linear in the entries of $S$, the resulting constraints are linear equalities.

\subsection*{SDP (SOS) problem} 
The search for the optimal conservative matrix correction can be formulated as a semidefinite program:
\[
\begin{aligned}
&\text{\textbf{Variables:}} && \text{Symmetric matrices } \Delta \in \mathbb{R}^{d\times d}, \; S \in \mathbb{R}^{M\times M},\\
&\text{\textbf{Maximize:}} && \Phi(\Delta) \quad (\text{e.g., } \operatorname{Tr}(\Delta)),\\
&\text{\textbf{Subject to:}} && S \succeq 0, \quad \Delta \succeq 0, \\
& && P_\Delta(v) \equiv z(v)^\top S z(v) \,(\text{coefficient matching for all monomials of degree 6}).
\end{aligned}
\tag{SOS-SDP}
\]
The choice of objective function $\Phi(\Delta)$ typically targets the ``largest" lower bound, such as the trace (sum of eigenvalues) or the log-determinant.

\subsection*{Correctness and proof of matrix bound}
\begin{theorem}[SDP certificate $\Rightarrow$ matrix bound]
\label{thm:sdp}
Let $\Delta \succeq 0$ and $S \succeq 0$ be feasible solutions for (SOS-SDP). Then, for every direction $v \in \mathbb{R}^d$:
\[
v^\top\Delta v \le v^\top(\Sigma-J^{-1})v.
\]
Consequently, $\Sigma \succeq J^{-1} + \Delta$.
\end{theorem}

\begin{proof}
By the feasibility of $S \succeq 0$, the polynomial $P_\Delta(v) = z(v)^\top S z(v)$ is non-negative for all $v$. By the definition of $P_\Delta(v)$:
\[
(d(v)^\top s(v))^2 \ge (v^\top\Delta v)(s(v)^\top G_{\mathcal{N}} s(v)).
\]
If $\|\Pi_v\|^2 = s(v)^\top G_{\mathcal{N}} s(v) > 0$, we may divide by the denominator to obtain:
\[
v^\top\Delta v \le \frac{(d(v)^\top s(v))^2}{s(v)^\top G_{\mathcal{N}} s(v)} = \mathcal{R}(v).
\]
By Theorem~\ref{thm:dir}, $\mathcal{R}(v) \le v^\top(\Sigma-J^{-1})v$, establishing the directional inequality. In the case where $s(v)^\top G_{\mathcal{N}} s(v) = 0$, we have $\Pi_v = 0$ and the directional bound \(\mathcal{R}(v)\) is $0$. Since $v^\top(\Sigma-J^{-1})v \ge 0$ (by Proposition~\ref{prop:proj_ineq}), the directional inequality holds trivially. Thus $\Sigma - J^{-1} \succeq \Delta$ as a matrix inequality.
\end{proof}

\begin{remark}[The Isotropic Case]
Consider a $d=2$ system where curvature is restricted to a single normal direction. Let $\tilde v = G^{-1} v$ and define:
\[
G_{\mathcal{N}} = c \begin{pmatrix} 1 & 0 & 0 \\ 0 & 0 & 0 \\ 0 & 0 & 0 \end{pmatrix}, \quad C^\top = \begin{pmatrix} a & b \\ 0 & 0 \\ 0 & 0 \end{pmatrix},
\]
for $c>0$. Here, $N(v) = (av_1 + bv_2)\tilde{v}_1^2$ and $D(v) = c\tilde{v}_1^4$. The directional bound becomes:
\[
\mathcal{R}(v) = \frac{(av_1 + bv_2)^2 \tilde{v}_1^4}{c \tilde{v}_1^4} = \frac{(av_1 + bv_2)^2}{c}.
\]

Since $\mathcal{R}(v)$ is a quadratic form, we have an exact matrix representation $\Delta = \frac{1}{c} \begin{pmatrix} a \\ b \end{pmatrix} \begin{pmatrix} a & b \end{pmatrix}$. In this scenario, $P_\Delta(v) \equiv 0$, and the (SOS-SDP) returns the exact matrix with the trivial certificate $S=0$. This demonstrates that the SOS-SDP framework recovers classical results, such as the Bhattacharyya-type bounds, whenever the manifold geometry is sufficiently symmetric.
\end{remark}

\section{Higher-order jet space refinements}
\label{sec:ho}
To further refine the variance bound, we consider the higher-order extrinsic geometry of the embedding $s(\Theta)$. This generalizes the second-order analysis by incorporating the local geometry of the $m$-th order jet space, allowing for a nested sequence of increasingly tight lower bounds. For a multi-index $\alpha = (\alpha_1, \dots, \alpha_d) \in \mathbb{N}_0^d$, recall the notation of the higher-order partial derivatives $\eta_\alpha = \partial^\alpha s(\theta) \in H$.
Also recall that the $m$-th order jet subspace $\mathcal{T}_m(\theta) \subset H$ is the span of all partial derivatives of the embedding up to total order $m$:
\[
\mathcal{T}_m = \operatorname{span} \{ \eta_\alpha : 1 \le |\alpha| \le m \}.
\]

\begin{definition}
We define the $m$-th order Bhattacharyya-type matrix bound $B^{(m)}$ by the inner products of the projections of the lifted errors onto this space:
\[
B^{(m)}_{pq} := \langle \operatorname{Proj}_{\mathcal{T}_m} \widetilde{Z}^{(p)}, \operatorname{Proj}_{\mathcal{T}_m} \widetilde{Z}^{(q)} \rangle.
\]
\end{definition}

By the nesting property $\mathcal{T}_1 \subset \mathcal{T}_2 \subset \dots \subset \mathcal{T}_m$, it follows immediately that $B^{(m+1)} \succeq B^{(m)}$, where $B^{(1)} = J^{-1}$ is the classical CRB. This sequence of matrices provides a discrete approximation to the total variance available in the higher-order score functions. We now extend the directional inequality in Theorem \ref{thm:dir} using a generalization of the second fundamental form.

\begin{definition}[Order-$(m+1)$ Directional Curvature]
For a direction $v \in \mathbb{R}^d$, let $\partial_v = \sum_{i=1}^d \tilde{v}_i \partial_i$ be the directional derivative operator where $\tilde{v} = G^{-1}v$. The $(m+1)$-th order directional curvature vector $\Pi_v^{(m+1)} \in \mathcal{T}_m^\perp$ is defined as the orthogonal projection of the $(m+1)$-th order directional derivative onto the complement of the previous jet space:
\[
\Pi_v^{(m+1)} := \operatorname{Proj}_{\mathcal{T}_m^\perp} \left( \partial_v^{m+1} s(\theta) \right).
\]
\end{definition}

\begin{theorem}[Higher-order directional curvature bound]
For every direction $v \in \mathbb{R}^d$ and jet order $m \ge 1$, under regularity in Assumption \ref{assump:smooth}, the estimator covariance $\Sigma$ satisfies:
\begin{equation}
\label{eq:high_order_dir}
v^\top (\Sigma - B^{(m)}) v \ge \frac{\langle Z_v, \Pi_v^{(m+1)} \rangle^2}{\|\Pi_v^{(m+1)}\|^2}.
\end{equation}
Equality holds if and only if the normal component of the lifted error $Z_v$ relative to $\mathcal{T}_m$ is collinear with $\Pi_v^{(m+1)}$.
\label{thm:hor}
\end{theorem}

\begin{proof}
Consider the orthogonal decomposition of the combined lifted error $Z_v$ with respect to $\mathcal{T}_m$:
\[
Z_v = \operatorname{Proj}_{\mathcal{T}_m} Z_v + R_m, \quad R_m \in \mathcal{T}_m^\perp.
\]
By the Pythagorean theorem in $H$, $\|Z_v\|^2 = \|\operatorname{Proj}_{\mathcal{T}_m} Z_v\|^2 + \|R_m\|^2$. From the definition of $B^{(m)}$, we recognize that $v^\top B^{(m)} v = \|\operatorname{Proj}_{\mathcal{T}_m} Z_v\|^2$. Thus, the remaining variance inflation is exactly $\|R_m\|^2$. 

Since $\Pi_v^{(m+1)}$ lies in $\mathcal{T}_m^\perp$ by definition, the Cauchy-Schwarz inequality in the subspace $\mathcal{T}_m^\perp$ implies:
\[
\|R_m\|^2 \ge \frac{\langle R_m, \Pi_v^{(m+1)} \rangle^2}{\|\Pi_v^{(m+1)}\|^2} = \frac{\langle Z_v, \Pi_v^{(m+1)} \rangle^2}{\|\Pi_v^{(m+1)}\|^2}\quad (=:\mathcal{R}_m(v)),
\]
where the final equality follows because $\operatorname{Proj}_{\mathcal{T}_m} Z_v \perp \Pi_v^{(m+1)}$.
\end{proof}

For \(m\geq 1\), $\Sigma\succeq B^{(m+1)} \succeq B^{(m)}$, and therefore we have the matrix correction \(\Sigma-B^{(m)}\succeq \Delta^{(m)}_{\mathrm{B}}\), where \(\Delta^{(m)}_{\mathrm{B}}:= B^{(m+1)} - B^{(m)}\). Let $\mathcal{N}_m = \mathcal{T}_{m+1} \cap \mathcal{T}_m^\perp$ be the $(m+1)$-th order normal space. The matrix-valued correction $\Delta^{(m)}_{\mathrm{B}} = B^{(m+1)} - B^{(m)}$ is calculated as:
\[
\Delta^{(m)}_{\mathrm{B}} = C_m G_{\mathcal{N},m}^{-1} C_m^\top,
\]
where $C_m$ is the matrix of pairings $c_{p,\alpha} = \langle \widetilde{Z}^{(p)}, \eta_\alpha \rangle$ for $\eta_\alpha \in \mathcal{N}_m$, and $G_{\mathcal{N},m}$ is the Gramian matrix of the basis for $\mathcal{N}_m$.

Although both \(v^\top \Delta^{(m)}_{\mathrm{B}}v\) and \(\mathcal{R}_m(v)\) (which is our bound in \eqref{eq:high_order_dir}) are homogeneous degree-2, they are generically different objects with \(\mathcal{R}_m(v)\) being a rational function. Furthermore, the bound in \eqref{eq:high_order_dir} respects the local directional geometry of the statistical manifold, whereas the global correction \(\Delta^{(m)}_{\mathrm{B}}\) does not. If we desire a matrix-valued correction that is consistent with the local extrinsic geometry, we may use the SOS-SDP formalism developed in Section \ref{sec:mc} for the \(m=1\) case.

\begin{remark}[SOS-SDP Extension]
The conservative matrix correction developed for $m=1$ generalizes to higher orders. For order $m$, the directional curvature $\Pi_v^{(m+1)}$ is a homogeneous polynomial in $v$ of degree $m+1$. Consequently, the directional bound $\mathcal{R}_m(v):=\frac{N_m(v)^2}{D_m(v)}$ is a rational function of degree $2(m+2) - 2(m+1) = 2$. A certified matrix $\Delta_m$ is then obtained via (SOS-SDP) by matching coefficients of the degree $2(m+2)$ polynomial $P_{\Delta, m}(v) = N_m(v)^2 - (v^\top \Delta_m v) D_m(v)$.
\end{remark}

\section{Two detailed examples}
\label{sec:ex}
We discuss two examples, illustrating different manifold geometries, aimed at comparing the directional bound \(\mathcal{R}(v)\), and the two matrix corrections (i) \(\Delta\) from the SOS-SDP framework and (ii) \(\Delta_\mathrm{B}\equiv \Delta^{(1)}_{\mathrm{B}}\).

\begin{example}[Gaussian location with curved third coordinate]
\label{ex:cg}
Consider the model
\[
X\sim N(\mu(\theta),\sigma^2 I_3),\qquad
\mu(\theta)=\begin{pmatrix}\mu_1\\[2pt]\mu_2\\[2pt]\mu_3\end{pmatrix}=\begin{pmatrix}\theta_1\\[2pt]\theta_2\\[2pt]\alpha\theta_1^2\end{pmatrix},\qquad
\theta=(\theta_1,\theta_2)\in\R^2,\ \alpha\neq0,
\]
with curvature correction evaluated at the convenient point \(\theta_1=0\). 

\noindent
\textbf{Square-root embedding and tangent vectors.}
The square-root embedding is
\[
s(\theta)(x)=(2\pi\sigma^2)^{-3/4}\exp\!\Big(-\frac{\|x-\mu(\theta)\|^2}{4\sigma^2}\Big).
\]
Use notation
\[
\mu_{,i}:=\begin{pmatrix}\tfrac{\partial\mu_1}{\partial \theta_i}\\[2pt]\tfrac{\partial\mu_2}{\partial \theta_i}\\[2pt]\tfrac{\partial\mu_3}{\partial \theta_i}\end{pmatrix},\qquad \mu_{,ij}:=\begin{pmatrix}\tfrac{\partial^2\mu_1}{\partial \theta_i \partial \theta_j}\\[2pt]\tfrac{\partial^2\mu_2}{\partial \theta_i \partial \theta_j}\\[2pt]\tfrac{\partial^2\mu_3}{\partial \theta_i \partial \theta_j}\end{pmatrix} .
\]
We have
\[
\eta_i(\theta)=\partial_i s(\theta)=\tfrac12 s(\theta)Y_i,\qquad
Y_i=\partial_i\log f.
\]
Then
\[
Y_i(x;\theta)=\frac{1}{\sigma^2}(x-\mu(\theta))\cdot\mu_{,i}(\theta),
\]
and it is convenient to set
\[
A_i(x):=\frac{1}{2\sigma^2}(x-\mu)\cdot\mu_{,i},
\]
so \( \eta_i(x)=s(x)A_i(x) \).

At \(\theta_1=0\), we have
\[
\mu_{,1}=(1,0,0)^\top,\quad \mu_{,2}=(0,1,0)^\top,\quad \mu_{,11}=(0,0,2\alpha)^\top.
\]
Hence
\[
J_{ij}=\frac{1}{\sigma^2}\mu_{,i}\cdot\mu_{,j}=\sigma^{-2}\delta_{ij},
\]
so \(J=\sigma^{-2}I_2,\ J^{-1}=\sigma^2 I_2\), and \(G=\tfrac14J,\ G^{-1}=4\sigma^2 I_2\).

\noindent
\textbf{Second derivatives and normal vectors.}
Differentiate
\[
\partial_i A_j(x)=\frac{1}{2\sigma^2}\big(-\mu_{,i}\cdot\mu_{,j} + (x-\mu)\cdot\mu_{,ij}\big),
\]
so
\[
\partial_i\partial_j s = s\big(A_iA_j + \partial_i A_j\big).
\]
At \(\theta_1=0\), we have \(\mu_{,12}=\mu_{,22}=0\) and only \(\mu_{,11}\) nonzero; the Christoffel inner products
\[
\langle\partial_i\partial_j s,\eta_m\rangle
= \frac{1}{4\sigma^2}\,\mu_{,ij}\cdot\mu_{,m}
\]
therefore vanish and we get \(\Gamma^\ell_{ij}=0\) at \(\theta_1=0\). Consequently, the second fundamental form vectors equal the second derivatives at this point:
\[
\Pi_{ij}=\partial_i\partial_j s \qquad i,j\in\{1,2\}.
\]

Compute these explicitly (write \(U=X_1-\mu_1,\ V=X_2-\mu_2,\ W=X_3-\mu_3\)):

\[
\begin{aligned}
A_1 &= \frac{U}{2\sigma^2},\qquad
A_2 = \frac{V}{2\sigma^2},\\[4pt]
\partial_1 A_1 &= \frac{1}{2\sigma^2}\big(-1 + 2\alpha W\big),\qquad
\partial_1 A_2 = 0,\qquad \partial_2 A_2 = -\frac{1}{2\sigma^2}.
\end{aligned}
\]

Hence the normal vectors are:
\[
\begin{aligned}
\Pi_{11}(x) &= s(x)\Big(\frac{U^2}{4\sigma^4} + \frac{1}{2\sigma^2}(-1 + 2\alpha W)\Big),\\[4pt]
\Pi_{12}(x) &= s(x)\frac{UV}{4\sigma^4},\\[4pt]
\Pi_{22}(x) &= s(x)\Big(\frac{V^2}{4\sigma^4} - \frac{1}{2\sigma^2}\Big).
\end{aligned}
\]

\noindent
\textbf{Estimator and lifted errors.}
Choose the unbiased estimator
\[
T^{(1)}(X)=X_1,\qquad T^{(2)}(X)=X_2+\gamma(X_3-\mu_3),
\qquad \gamma\in\R,
\]
so the lifted (centered) errors are
\[
\widetilde Z^{(1)}=(X_1-\theta_1)s = U s,\qquad
\widetilde Z^{(2)}=(V+\gamma W)s.
\]

\noindent
\textbf{Numerator and denominator in curvature correction term.}
We compute all pairings needed for the numerator in the curvature correction in Theorem \ref{thm:dir}. Using the basic facts that odd moments vanish and independent coordinates factor, we obtain:
\[
\begin{aligned}
\langle\widetilde Z^{(1)},\Pi_{11}\rangle
&= \E\Big[U\Big(\frac{U^2}{4\sigma^4}+\frac{1}{2\sigma^2}(-1+2\alpha W)\Big)\Big] = 0,\\[4pt]
\langle\widetilde Z^{(2)},\Pi_{11}\rangle
&= \E\Big[(V+\gamma W)\Big(\frac{U^2}{4\sigma^4}+\frac{1}{2\sigma^2}(-1+2\alpha W)\Big)\Big]
= \gamma\alpha,
\end{aligned}
\]
because the only nonvanishing contribution is \(\gamma\cdot(2\alpha/(2\sigma^2))\E[W^2]=\gamma\alpha\).

For \(\Pi_{12}\), it is easily seen that
\[
\langle\widetilde Z^{(p)},\Pi_{12}\rangle
=  0
\quad\text{for }p=1,2,
\]
by oddness/independence.

For \(\Pi_{22}\):
\[
\langle\widetilde Z^{(1)},\Pi_{22}\rangle
= \E\big[U(\tfrac{V^2}{4\sigma^4}-\tfrac{1}{2\sigma^2})\big] = 0,
\]
and
\[
\langle\widetilde Z^{(2)},\Pi_{22}\rangle
= \E\big[(V+\gamma W)(\tfrac{V^2}{4\sigma^4}-\tfrac{1}{2\sigma^2})\big] = 0,
\]
again by oddness and independence.

For this estimator the only nonzero pairing is
\[
\langle\widetilde Z^{(2)},\Pi_{11}\rangle = \gamma\alpha,
\]
hence in 
\[
\langle Z_v,\Pi_v\rangle
= \sum_{p=1}^2 v_p \sum_{i,j=1}^2 \widetilde v_i\widetilde v_j \langle\widetilde Z^{(p)},\Pi_{ij}\rangle,
\]
only the term with \(p=2\) and \((i,j)=(1,1)\) contributes. Therefore
\[
\langle Z_v,\Pi_v\rangle \;=\; v_2\,\widetilde v_1^2\,(\gamma\alpha).
\]

Thus the numerator equals \((v_2\widetilde v_1^2\gamma\alpha)^2\).

Although \(\Pi_{12},\Pi_{22}\) do not enter the numerator, they \emph{do} affect the denominator \(\|\Pi_v\|^2\).
It is not hard to see that
\[
\begin{aligned}
\|\Pi_{11}\|^2 &= \E\big[(A_1^2+\partial_1 A_1)^2\big]
= \frac{3}{16\sigma^4} + \frac{\alpha^2}{\sigma^2},\\[6pt]
\|\Pi_{12}\|^2 &= \E[(A_1A_2)^2] = \frac{1}{16\sigma^4},\\[6pt]
\|\Pi_{22}\|^2 &= \E\big[(A_2^2+\partial_2 A_2)^2\big] = \frac{3}{16\sigma^4},
\end{aligned}
\]
and the nonzero cross-term
\[
\langle\Pi_{11},\Pi_{22}\rangle
= \frac{1}{16\sigma^4},
\]
whereas \(\langle\Pi_{11},\Pi_{12}\rangle=\langle\Pi_{12},\Pi_{22}\rangle=0\).  

By definition,
\[
\Pi_v = \sum_{i,j=1}^2 \widetilde v_i\widetilde v_j\,\Pi_{ij},
\qquad \widetilde v = G^{-1} v = 4\sigma^2 v,
\]
and the coordinates of \(\Pi_v\) in the basis \((\Pi_{11},\Pi_{12},\Pi_{22})\) are
\[
\big(\widetilde v_1^2,\; 2\widetilde v_1\widetilde v_2,\; \widetilde v_2^2\big).
\]

Thus we get the closed form
\[
\|\Pi_v\|^2
= \frac{3}{16\sigma^4}(\widetilde v_1^2+\widetilde v_2^2)^2
+ \frac{\alpha^2}{\sigma^2}\,\widetilde v_1^4.
\]

\noindent
\textbf{Final directional bound.}
The directional curvature-corrected CRB (Theorem~\ref{thm:dir}) gives:
\[
v^\top(\Sigma-J^{-1})v \;\ge\; 
\frac{\langle Z_v,\Pi_v\rangle^2}{\|\Pi_v\|^2}
= \frac{\big(v_2\,\widetilde v_1^2\,\gamma\alpha\big)^2}{\|\Pi_v\|^2}.
\]

Substitute \(\widetilde v_i=4\sigma^2 v_i\). After canceling the common factor \(16\sigma^4\), the bound simplifies to the compact rational form
\[
v^\top(\Sigma-J^{-1})v \;\ge\;
\frac{16\sigma^4\,v_2^2\,v_1^4\,\gamma^2\alpha^2}
{\,3(v_1^2+v_2^2)^2 + 16\sigma^2\alpha^2 v_1^4\, }.
\]
\end{example}

It is instructive to dig deeper into the geometric aspects of the Gaussian example with further evaluations to assess how finer directional information makes \(\mathcal{R}(v)\) and the SOS-SDP correction different from classical matrix-level refinements such as \(\Delta_{\mathrm{B}}\).

\subsection{Numerical evaluation and the geometry of the curvature bound}

We investigate the tightness of our proposed curvature corrections using the curved Gaussian location model just analyzed. The model parameters are fixed at $\sigma^2 = 1$ and $\alpha = 0.5$, with the remaining factors as in Example \ref{ex:cg}. Recall the concrete estimator we considered:
\[
T^{(1)}(X) = X_1, \quad T^{(2)}(X) = X_2 + \gamma(X_3 - \mu_3),
\]
with the coupling parameter $\gamma = 1.5$. This estimator is specifically designed to interact with the manifold's extrinsic curvature in the $X_3$ dimension.

\noindent
\textbf{The directional bound and the pinching effect.}
The directional curvature-corrected bound, $\mathcal{R}(v)$, represents the local geometric limit of estimator efficiency. As derived in the example, the bound for this model takes a specific rational form:
\[
\mathcal{R}(v) = \frac{16\sigma^4 v_2^2 v_1^4 \gamma^2 \alpha^2}{3(v_1^2 + v_2^2)^2 + 16\sigma^2 \alpha^2 v_1^4}.
\]
A critical feature of this bound is revealed upon inspection of the coordinate axes. Due to the terms $v_2^2$ and $v_1^4$ in the numerator, the bound vanishes exactly when $v_1 = 0$ or $v_2 = 0$. Along the principal axes, the vanishing of $\mathcal{R}(v)$ indicates that the manifold's extrinsic curvature does not impose a second-order variance penalty, allowing the estimator to saturate the first-order Fisher limit. Conversely, in oblique directions, the manifold's `twist' induces a non-zero $\mathcal{R}(v)$; here, any deviation of the estimator's variance from $J^{-1} + \mathcal{R}(v)$ represents a measure of second-order inefficiency. Importantly, the vanishing of \(\mathcal{R}(v)\) along the axes creates a ``pinched'' clover-leaf structure in the polar representation of the variance, a phenomenon we term the \textit{pinching effect} (cf. Figure \ref{fig:pinching}).

\noindent
\textbf{SOS-SDP matrix correction.}
We seek a symmetric PSD matrix $\Delta$ that provides a global certificate for the variance gap, satisfying the matrix inequality $\Sigma \succeq J^{-1} + \Delta$ while also respecting the local manifold geometry; that is, $v^\top \Delta v \le \mathcal{R}(v)$ for all $v \in \mathbb{R}^d$. To achieve this, one may solve the (SOS-SDP) problem by matching the coefficients of the degree-6 polynomial:
\[
P_\Delta(v) := \mathcal{R}(v)\|\Pi_v\|^2 - (v^\top \Delta v)\|\Pi_v\|^2 \ge 0.
\]
Numerically, for the specified Gaussian model, the (SOS-SDP) returns $\Delta \approx 0$. This result, while seemingly conservative, is mathematically necessary. Because $\mathcal{R}(v)$ pinches to zero along the axes, any quadratic form $v^\top \Delta v$ (which appears as an ellipse in polar coordinates) that remains globally ``underneath'' $\mathcal{R}(v)$ must itself have zero eigenvalues. This is a consequence of the fact that (SOS-SDP) looks for the largest matrix $\Delta$ that represents the universal floor defined by the geometry of the manifold, respecting the fact that this floor drops to the CRB on the axes. This highlights a fundamental finding: for curved models, the directional bound $\mathcal{R}(v)$ could be a strictly more informative tool than any global matrix-form bound.

\begin{figure}[ht]
    \centering
    \includegraphics[width=0.6\textwidth]{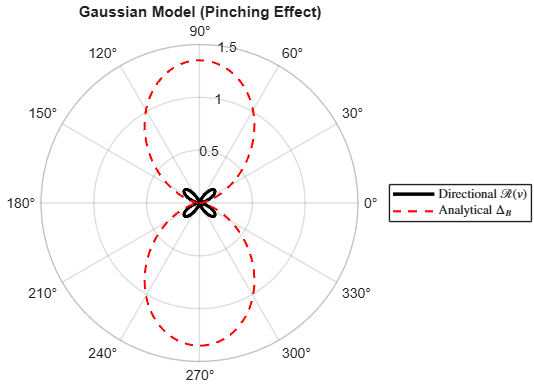}
    \caption{The ``pinching effect" in the Gaussian location model. The directional bound $\mathcal{R}(v)$ (black solid) vanishes along the coordinate axes, revealing the non-quadratic nature of the curvature correction. The classical analytical $\Delta_\mathrm{B}$ matrix (red dashed) violates the directional bound in these regions, demonstrating that a global matrix correction can be overly optimistic. This necessitates the use of our SOS-SDP approach to find a truly conservative, certified matrix $\Delta$.}
    \label{fig:pinching}
\end{figure}

\noindent
\textbf{Comparison with the analytical $\Delta_{\mathrm{B}}$ matrix.}
The classical second-order Bhattacharyya matrix $B^{(2)}$ provides a statistical benchmark, representing the variance captured by the second-order jet space $\mathcal{T}_2$. For our Gaussian model, we consider the global matrix correction $\Delta^{(1)}_{\mathrm{B}}\equiv\Delta_{\mathrm{B}} = B^{(2)} - B^{(1)}$ defined in Section \ref{sec:ho}. Since the Christoffel symbols $\Gamma_{ij}^k$ vanish at the test point, the second derivatives $\eta_{ij} = \partial_i \partial_j s(\theta)$ are already orthogonal to the tangent space $\mathcal{T}_1$. This allows us to compute the projection onto the normal space $\mathcal{N}_1 = \operatorname{span}\{\eta_{11}, \eta_{12}, \eta_{22}\}$ directly.

The matrix correction is computed as $\Delta_{\mathrm{B}} = C G_{\mathcal{N}}^{-1} C^\top$, where $C$ is the matrix of pairings $c_{p,\alpha} = \langle \widetilde{Z}^{(p)}, \eta_\alpha \rangle$. For the specified estimator $T$, the only non-vanishing pairing is $c_{2,11} = \gamma\alpha$, as already seen in Example \ref{ex:cg}. The normal Gramian $G_{\mathcal{N}}$, representing the metric of the second-order scores, is given by:
\[
G_{\mathcal{N}} = \begin{pmatrix} \frac{3}{16\sigma^4} + \frac{\alpha^2}{\sigma^2} & 0 & \frac{1}{16\sigma^4} \\ 0 & \frac{2}{16\sigma^4} & 0 \\ \frac{1}{16\sigma^4} & 0 & \frac{3}{16\sigma^4} \end{pmatrix}.
\]
By inverting $G_{\mathcal{N}}$, we find the second-order variance inflation matrix:
\[
\Delta_{\mathrm{B}} = \begin{pmatrix} 0 & 0 \\ 0 & \frac{6\sigma^4 \gamma^2 \alpha^2}{1 + 6\sigma^2 \alpha^2} \end{pmatrix}.
\]
For the parameters $\sigma^2=1, \alpha=0.5, \gamma=1.5$, we obtain the scalar bound $\Delta_{\mathrm{B}, 22} = 1.35$. 

It is instructive to reconcile why this analytical bound is numerically larger than the directional bound. The Bhattacharyya gap represents the squared norm of the estimator error's projection onto the \textit{entire} normal subspace $\mathcal{N}_1$. Letting $\{\mathbf{n}_a\}$ be an orthonormal basis for $\mathcal{N}_1$, we have:
\[
v^\top (\Sigma - B^{(1)}) v \ge v^\top \Delta_{\mathrm{B}} v = \sum_{a} \langle Z_v, \mathbf{n}_a \rangle^2 \ge \frac{\langle Z_v, \Pi_v \rangle^2}{\|\Pi_v\|^2} = \mathcal{R}(v).
\]

The inequality follows from the Bessel inequality in the normal subspace $\mathcal{N}_1$. While $\Delta_{\mathrm{B}}$ provides an algebraically higher lower bound by aggregating all curvature components in the jet space, it remains direction-blind. Mathematically, the matrix gap $v^\top \Delta_{\mathrm{B}} v$ measures the total squared norm of the error's projection onto the three-dimensional subspace $\mathcal{N}_1$. In contrast, $\mathcal{R}(v)$ represents the squared projection onto the specific $1$-dimensional span of $\Pi_v \subseteq \mathcal{N}_1$. To reiterate what we just mentioned:
\[
v^\top \Delta_{\mathrm{B}} v = \|\operatorname{Proj}_{\mathcal{N}_1} Z_v\|^2 \ge \|\operatorname{Proj}_{\operatorname{span}\{\Pi_v\}} Z_v\|^2 = \mathcal{R}(v).
\]

Specifically, along the axis \(v_1=0\), $\mathcal{R}(v)$ vanishes due to the quartic pinching of the error projection in the direction of the curvature vector $\Pi_v$, whereas $\Delta_{\mathrm{B}}$ maintains a constant quadratic contribution. On the other hand, recall what we saw above: Our SOS-SDP identifies that because the clover-leaf $\mathcal{R}(v)$ hits zero at the axes, no quadratic form (ellipse) can globally lower-bound it without its eigenvalues also vanishing. This forces the conservative result $\Delta \approx 0$, revealing that the Bhattacharyya-type bound overestimates the universal variance floor in directions where the manifold's curvature effect is quartically suppressed. The SOS-SDP correctly identifies that no strictly positive-definite matrix can consistently certify the information gap across the entire parameter space; see Table \ref{tab:results}.

\begin{table}[ht]
\centering
\caption{Comparison of Variance Lower Bounds ($\sigma^2=1, \alpha=0.5, \gamma=1.5$)}
\label{tab:results}
\begin{tabular}{@{}lccc@{}}
\toprule
\textbf{Metric} & \textbf{Analytical $\Delta_{\mathrm{B}}$} & \textbf{SDP (Conservative)} & \textbf{Directional $\mathcal{R}(v_{\text{max}})$} \\ \midrule
$\Delta_{11}$ & 0.0000 & 0.0000 & 0.0000 \\
$\Delta_{22}$ & 1.3500 & 0.0000 & 0.14 (at $\phi \approx 55^\circ$) \\
\textbf{Integrity} & Subspace-Optimistic & Rigorously Certified & Locally Tight \\ \bottomrule
\end{tabular}
\end{table}

\noindent
\textbf{Summary of results.}
The discrepancy between $\Delta_{\mathrm{B}}$ and $\mathcal{R}(v)$ demonstrates that the classical second-order matrix bound can be viewed as an isotropic approximation that may over-promise variance inflation in degenerate directions due to the lack of finer directional information. We conclude that the directional curvature bound $\mathcal{R}(v)$ is the more fundamental metric for analyzing nonasymptotic, local efficiency. While $\Delta_{\mathrm{B}}$ is a valid lower bound for the restricted class of estimators linear in the second-order scores, $\mathcal{R}(v)$ captures the true geometric limit imposed by the manifold's extrinsic curvature, which standard quadratic matrix approximations cannot maintain in the presence of the pinching effect. A slightly more technical discussion on the implications of the (spherical) square-root geometry on our directional bounds can be found in Section \ref{sec:syn}.

\noindent
\textbf{Implications for adaptive estimation.}
The identification of the pinching effect has immediate consequences for the design of computationally efficient estimators in curved families. Our analysis suggests that the higher-order efficiency loss is not an isotropic penalty but is instead highly sensitive to the direction of the parameter update. 

Specifically, in update directions along which the directional bound $\mathcal{R}(v)$ vanishes despite non-zero extrinsic curvature, an estimator may achieve first-order efficiency (approaching the CRB) even with a simple linear structure. Conversely, in the oblique directions where $\mathcal{R}(v)$ is maximized, nonlinear curvature correction terms are indispensable. This geometric insight suggests an \textit{adaptive correction} strategy: applying curvature-aware adjustments only in parameter subspaces where the directional sensitivity $\mathcal{R}(v)/\|v\|^2$ exceeds a predefined threshold. This approach preserves the robustness of the estimator while minimizing the computational overhead of higher-order derivative calculations.

\begin{example}[Spherical multinomial model]

To demonstrate a case where the (SOS-SDP) yields a non-trivial, full-rank matrix correction, we consider a multinomial model with three outcomes $\mathcal{X} = \{x_1, x_2, x_3\}$. The square-root embedding $s(\theta) \in L^2(\mathcal{X}, \mu)$ maps to a 2D spherical cap on the unit 2-sphere in $\mathbb{R}^3$, where $\mu$ is the counting measure. Let the embedding be parameterized as:
\[
s(\theta) = \begin{pmatrix} \cos \theta_1 \cos \theta_2 \\ \sin \theta_1 \cos \theta_2 \\ \sin \theta_2 \end{pmatrix}, \quad \theta \in (-\epsilon, \epsilon)^2.
\]
At the point $\theta = (0,0)$, the tangent vectors (score functions) are $\eta_1 = (0, 1, 0)^\top$ and $\eta_2 = (0, 0, 1)^\top$. The FIM is $J_{ij} = 4 \langle \eta_i, \eta_j \rangle = 4 \delta_{ij}$, or $J = 4 I_2$. The normal space $\mathcal{N}_1$ is spanned by the radial unit vector $\mathbf{n} = (-1, 0, 0)^\top$. The second-order geometry is captured by the derivatives $\eta_{ij} := \partial_i \partial_j s(\theta)$. For this spherical embedding, the second fundamental form components (the projections of $\eta_{ij}$ onto $\mathcal{N}_1$) are $\Pi_{11} = \Pi_{22} = \mathbf{n}$ and $\Pi_{12} = 0$.

To realize the variance inflation predicted by the curvature bound, we consider a test error vector $Z_v \in H$ that satisfies the first-order efficiency condition $\langle Z_v, \eta_i \rangle = \frac{1}{2} v_i$, but possesses a non-vanishing projection onto the normal space $\mathcal{N}_1$.\\
\textbf{Remark on unbiasedness}: It is important to note that a strictly unbiased estimator must satisfy $\langle Z_v, s \rangle = 0$; that is, the lifted error is necessarily ``radially blind" to the spherical curvature imposed ny the square-root geometry. In this 3-outcome multinomial model, $\mathcal{N}_1 = \text{span}\{s\}$, meaning no transverse (non-radial) normal directions are available for an unbiased estimator to couple with. Despite this constraint, we utilize this example as a geometric benchmark for isotropic coupling. By ``injecting" a radial component $\gamma$, we test the SOS-SDP's capacity to recover a full-rank matrix correction $\Delta$ when the error is uniformly distributed across the normal bundle. The implications of this spherical constraint and the resulting distinction between radial and transverse curvature are discussed in detail in Section \ref{sec:syn}.

Specifically, the error vector is characterized by a coupling parameter $\gamma$, such that the directional lifted error decomposes as: 
$$Z_v = 2\sum_{j,k=1}^2 v_j [J^{-1}]_{jk} \eta_k + \gamma \frac{\Pi_v}{\|\Pi_v\|},$$ 
where the first term is the unique minimum-norm vector in the tangent space $\mathcal{T}_1$ satisfying the score constraints. By the orthogonality of the jet space decomposition $\mathcal{T}_2 = \mathcal{T}_1 \oplus \mathcal{N}_1$, the total variance in direction $v$ is given by the squared norm:
$$\operatorname{Var}_{\theta}(v^\top T) = \|Z_v\|^2 = \|\operatorname{Proj}_{\mathcal{T}_1} Z_v\|^2 + \gamma^2.$$
Recognizing that the tangent projection norm corresponds to the classical Cramer-Rao limit, $\|\operatorname{Proj}_{\mathcal{T}_1} Z_v\|^2 =  v^\top J^{-1} v$, the directional curvature bound (Theorem \ref{thm:dir}) simplifies to:$$\mathcal{R}(v) = \frac{\langle Z_v, \Pi_v \rangle^2}{\|\Pi_v\|^2} = \frac{(\gamma \|\Pi_v\|)^2}{\|\Pi_v\|^2} = \gamma^2.$$

For the spherical case, the SOS-SDP seeks a matrix $\Delta \succeq 0$ satisfying the global certificate $v^\top \Delta v \le \mathcal{R}(v)$ for all $\|v\|=1$. On the unit sphere of directions, this obviously yields the non-degenerate solution $\Delta = \gamma^2 I_2$. Numerical results confirm that for $\gamma = 1.2$, the certified matrix bound is $\Delta \approx 1.44 I_2$. Moreover, since the unexplained part of the variance is completely along the direction \(\Pi_v\), \(v^\top\Delta_{\mathrm{B}}v\) also coincides with \(\mathcal{R}(v)\) and \(v^\top \Delta v\); see Figure \ref{fig:sphere_isotropic}. In fact, at the test point, \(\mathcal{N}_1=\mathrm{span}\{\Pi_v\}\). This demonstrates that in isotropic geometries, the SOS-certified bound successfully recovers the full-rank curvature correction, contrasting with the degenerate pinching observed in the Gaussian case.
\end{example}

\subsection{Synthesis of results: Degenerate vs. Non-degenerate geometry}
\label{sec:syn}
The contrast between the curved Gaussian location model and the spherical multinomial model illustrates the fundamental role of the coupling between manifold geometry and the estimator in determining the structure and attainability of higher-order variance bounds. For a pictorial comparison of both side-by-side, refer to Figure \ref{fig:comparison}.

In the Gaussian case, the directional bound $\mathcal{R}(v)$ exhibits a pinching effect, vanishing along the principal axes. This represents a highly structured, anisotropic curvature where the extrinsic ``bending" of the manifold does not impede estimation efficiency in all directions equally. Mathematically, the rational quartic nature of $\mathcal{R}(v)$ allows it to touch the CRB at specific points. In such geometries, the SOS-SDP correctly identifies that no non-zero quadratic form (ellipse) can globally lower-bound the variance gap without violating these local (pinched) geometric limits. This forces a distinction between the subspace-optimistic Bhattacharyya difference matrix $\Delta_\mathrm{B}$ and the true geometric limit $\mathcal{R}(v)$, with the latter being the only locally tight certificate.

In contrast, the spherical multinomial model possesses isotropic curvature. While all square-root density embeddings $s(\theta)$ are constrained to the unit sphere in $H$ due to the normalization $\int f d\mu = 1$, the multinomial case represents a symmetric embedding where the manifold bends uniformly relative to the tangent directions.

As illustrated in the results, for the multinomial model, the directional bound $\mathcal{R}(v)$, the certified SOS-SDP matrix $\Delta$, and the analytical second-order Bhattacharyya difference matrix $\Delta_\mathrm{B}$ all coincide exactly. This alignment occurs because the symmetry of the spherical embedding ensures that the ``information gap" is distributed uniformly across all parameter directions. Thus, the SOS-SDP does not merely approximate a bound; it recovers the exact deficit imposed by the spherical embedding.
In other words, this isotropic behavior serves as a geometric benchmark: here, the Hilbert space embedding $s(\Theta)$ exhibits constant extrinsic curvature, and for an estimator with uniform coupling $\gamma$, the SOS-SDP recovers a full-rank correction $\Delta = \gamma^2 I$. The vanishing $\Delta$ observed in our Gaussian model is therefore not a failure of the SOS-SDP, but a rigorous certification of the anisotropic `pinching' inherent to the model's embedding geometry, which coordinate-local bounds like Bhattacharyya's fail to capture. 

\noindent
\textbf{The universal spherical constraint.} The fact that $s(\theta)$ is a radial normal vector to the second-order geometry is a universal property of statistical manifolds. Since $\|s\|^2 = 1$, we differentiate with respect to $\theta_k$:
$$2\langle s, \partial_k s \rangle = 0 \implies s \perp \eta_k.$$
Differentiating again with respect to $\theta_j$ yields:
$$\langle \partial_j s, \partial_k s \rangle + \langle s, \partial_j \partial_k s \rangle = 0.$$
Substituting the Fisher information components $J_{jk} = 4\langle \eta_j, \eta_k \rangle$, we find:
\begin{equation}
\label{eq:iw}
\langle s, \eta_{jk} \rangle = -\frac{1}{4} J_{jk}.
\end{equation}
For the isotropic multinomial model, the second fundamental form $\Pi_{jk}$ is dominated by this radial component. However, the directional bound $\mathcal{R}(v)$ only captures the portion of curvature that the estimator's error $Z_v$ can `see'. 
For any direction $v$, we can write the directional curvature vector as:
$$\Pi_v = \langle \Pi_v, s \rangle s + \Pi_v^{\perp s}.$$
Using the identity we derived in \eqref{eq:iw}, $\langle s, \Pi_v \rangle = -\frac{1}{4} v^\top J v$. Thus:
$$\Pi_v = \underbrace{-\frac{1}{4} (v^\top J v) s}_{\text{Radial/Spherical}} + \underbrace{\Pi_v^{\perp s}}_{\text{Transverse/Twist}}.$$
The divergence between the two examples is fundamentally rooted in the radial alignment of the estimator's error in the Hilbert space. Because an unbiased estimator must satisfy $\mathbb{E}_\theta[T] = \theta$, it follows that the lifted error $Z_v = (T-\theta)s$ is always orthogonal to the radial vector: $\langle Z_v, s \rangle = 0$. Consequently, all unbiased estimators are ``radially blind" to the universal spherical curvature $J/4$. This blindness has a significant implication for the 3-outcome multinomial benchmark. Since the embedding $s(\Theta)$ is a spherical cap, the second fundamental form is purely radial, satisfying $\Pi_{jk} = -\frac{1}{4} J_{jk} s$. In this specific geometry, the transverse component vanishes ($\Pi_v^{\perp s} = 0$), which would imply $\mathcal{R}(v) = 0$ for any strictly unbiased estimator. Our choice to `inject' a coupling parameter $\gamma$ into the multinomial error vector $Z_v$ is therefore a deliberate theoretical construction. It allows us to utilize the multinomial model as a benchmark for isotropic normal information, effectively testing the SOS-SDP's ability to recover a full-rank matrix correction $\Delta$ when the manifold's curvature is uniform, even if that curvature is purely radial.

In actuality, the information gap $\mathcal{R}(v)$ is therefore determined solely by the coupling with the transverse curvature $\Pi_v^{\perp s}$. In our Gaussian model, this transverse component is highly anisotropic and specifically tied to the $v_2 \tilde{v}_1^2$ parameter interaction, resulting in the characteristic quartic pinching at the axes. In contrast, in the 3-outcome multinomial case, the normal space is purely radial, $\mathcal{N}_1 = \text{span}\{s\}$. As noted already in the last paragraph, a strictly unbiased estimator is radially blind ($\langle Z_v, s \rangle = 0$), but this example serves as a benchmark for isotropic coupling. By considering a lifted error $Z_v$ with a constant projection $\gamma$ onto the normal space, we demonstrate that the SOS-SDP recovers the full-rank matrix $\Delta = \gamma^2 I$ whenever the curvature and error alignment are uniform. In contrast, the Gaussian model possesses a higher-dimensional normal space where the estimator couples with transverse (non-radial) curvature. The `pinching' there arises because that transverse curvature is anisotropic, unlike the idealized isotropic coupling of the spherical benchmark.

This synthesis validates the SOS-SDP framework as a robust generalization of higher-order estimation theory, providing safe and rigorously certified estimates in settings where traditional matrix benchmarks may be overly optimistic.

\begin{figure}[ht]
    \centering
    \begin{subfigure}[b]{0.48\textwidth}
        \centering
        \includegraphics[width=\textwidth]{sg.png}
        \caption{Gaussian Model: The Pinching Effect.}
        \label{fig:gauss_pinched}
    \end{subfigure}
    \hfill
    \begin{subfigure}[b]{0.48\textwidth}
        \centering
        \includegraphics[width=\textwidth]{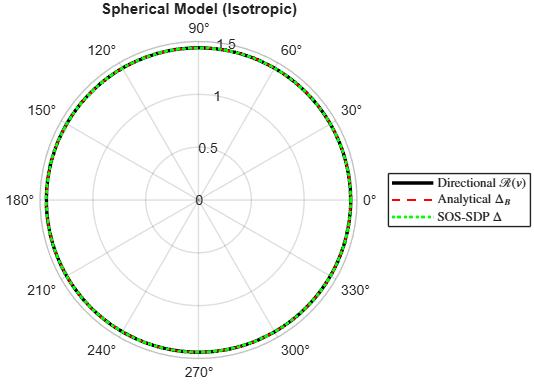}
        \caption{Spherical Model: Isotropic Curvature.}
        \label{fig:sphere_isotropic}
    \end{subfigure}
    \caption{Comparison of directional and matrix bounds across differing manifold geometries. In the Gaussian case (a), the analytical $\Delta_\mathrm{B}$ violates the directional bound $\mathcal{R}(v)$ at the axes, leading to $\Delta=0$ in the (SOS-SDP). In the spherical case (b), all bounds coincide.}
    \label{fig:comparison}
\end{figure}

\section{Conclusion and future work}

In this work, we developed nonasymptotic extrinsic curvature corrections to the CRB for vector parameters, extending previous scalar results to the multivariate setting. By utilizing the square-root embedding $s(\theta)$, we demonstrated that the extrinsic geometry of the model manifold provides a rigorous foundation for deriving higher-order variance bounds. 

A primary contribution of this framework is the derivation of the directional curvature-corrected bound $\mathcal{R}(v)$, which captures the second-order structure of the model's normal bundle to provide a high-fidelity representation of estimation limits. Our analysis reveals a fundamental dichotomy in multivariate curvature effects. In the curved Gaussian location model, we identified a pinching effect where the directional bound vanishes along principal axes. In such cases, we showed that classical benchmarks using the second-order Bhattacharyya matrix represent subspace-level variance which may decouple from the directional curvature $\Pi_v$. Consequently, the Bhattacharyya difference matrix $\Delta_\mathrm{B}=B^{(2)}-B^{(1)}$ fails to provide valid global certificates for general estimators, while our SOS-SDP framework yields a robust, conservative matrix correction $\Delta$. Conversely, the spherical multinomial model demonstrates that in isotropic geometries, the SOS-certified bound $\Delta$ successfully recovers the classical second-order limits, aligning $\mathcal{R}(v)$, $\Delta_{\mathrm{B}}$, and the SDP solution exactly. This isotropic case serves as a geometric benchmark for models where the curvature and estimator coupling are uniformly distributed across the parameter space. In contrast, the Gaussian `pinching' identifies the structural anisotropy of the manifold's transverse (non-radial) curvature, whose coupling with the estimator error vanishes along the principal score axes. This distinction reveals that the SOS-SDP does not merely approximate a bound, but rigorously identifies whether the manifold’s `twist' occurs in dimensions accessible to the estimator.
These results underscore that a faithful multivariate CRB refinement must account for the directional topology of the manifold. 

As in our work with scalar parameters, a possible drawback is that numerical computation of our bounds may be slightly trickier (but definitely within manageable complexity, as seen in the examples) than the classical ones using simpler score functions, but this increased complexity of the square-root scores is also what helps us capture parts of the estimator variance missed by those bounds. Furthermore, our motivation here was also largely theoretical, seeking to use extrinsic geometry of the statistical manifold to explain estimator inefficiency.

Several directions remain open for future study. An immediate extension is to adapt the present analysis to Bayesian or restricted-bias settings, where the interaction between prior information and manifold curvature may yield tighter nonasymptotic bounds. It would also be valuable to investigate conditions under which our bounds using higher-order jet-space corrections ($\mathcal{T}_m$ for $m > 2$) are attainable by concrete statistical estimators. Finally, connecting this extrinsic Hilbert-space perspective with information-geometric approaches to intrinsic CRBs may provide a unified understanding of how normal curvature effects — governed by the second and higher-order fundamental forms — interact with the intrinsic Riemannian structure of the statistical manifold.

\bibliography{sn-bibliography}


\begin{thebibliography}{18}
\ifx \bisbn   \undefined \def \bisbn  #1{ISBN #1}\fi
\ifx \binits  \undefined \def \binits#1{#1}\fi
\ifx \bauthor  \undefined \def \bauthor#1{#1}\fi
\ifx \batitle  \undefined \def \batitle#1{#1}\fi
\ifx \bjtitle  \undefined \def \bjtitle#1{#1}\fi
\ifx \bvolume  \undefined \def \bvolume#1{\textbf{#1}}\fi
\ifx \byear  \undefined \def \byear#1{#1}\fi
\ifx \bissue  \undefined \def \bissue#1{#1}\fi
\ifx \bfpage  \undefined \def \bfpage#1{#1}\fi
\ifx \blpage  \undefined \def \blpage #1{#1}\fi
\ifx \burl  \undefined \def \burl#1{\textsf{#1}}\fi
\ifx \doiurl  \undefined \def \doiurl#1{\url{https://doi.org/#1}}\fi
\ifx \betal  \undefined \def \betal{\textit{et al.}}\fi
\ifx \binstitute  \undefined \def \binstitute#1{#1}\fi
\ifx \binstitutionaled  \undefined \def \binstitutionaled#1{#1}\fi
\ifx \bctitle  \undefined \def \bctitle#1{#1}\fi
\ifx \beditor  \undefined \def \beditor#1{#1}\fi
\ifx \bpublisher  \undefined \def \bpublisher#1{#1}\fi
\ifx \bbtitle  \undefined \def \bbtitle#1{#1}\fi
\ifx \bedition  \undefined \def \bedition#1{#1}\fi
\ifx \bseriesno  \undefined \def \bseriesno#1{#1}\fi
\ifx \blocation  \undefined \def \blocation#1{#1}\fi
\ifx \bsertitle  \undefined \def \bsertitle#1{#1}\fi
\ifx \bsnm \undefined \def \bsnm#1{#1}\fi
\ifx \bsuffix \undefined \def \bsuffix#1{#1}\fi
\ifx \bparticle \undefined \def \bparticle#1{#1}\fi
\ifx \barticle \undefined \def \barticle#1{#1}\fi
\bibcommenthead
\ifx \bconfdate \undefined \def \bconfdate #1{#1}\fi
\ifx \botherref \undefined \def \botherref #1{#1}\fi
\ifx \url \undefined \def \url#1{\textsf{#1}}\fi
\ifx \bchapter \undefined \def \bchapter#1{#1}\fi
\ifx \bbook \undefined \def \bbook#1{#1}\fi
\ifx \bcomment \undefined \def \bcomment#1{#1}\fi
\ifx \oauthor \undefined \def \oauthor#1{#1}\fi
\ifx \citeauthoryear \undefined \def \citeauthoryear#1{#1}\fi
\ifx \endbibitem  \undefined \def \endbibitem {}\fi
\ifx \bconflocation  \undefined \def \bconflocation#1{#1}\fi
\ifx \arxivurl  \undefined \def \arxivurl#1{\textsf{#1}}\fi
\csname PreBibitemsHook\endcsname

\bibitem[\protect\citeauthoryear{Krishnan}{2025}]{srk}
\begin{botherref}
\oauthor{\bsnm{Krishnan}, \binits{S.R.}}:
Improving Cram\'er-Rao Bound And Its Variants: An Extrinsic Geometry
  Perspective
(2025).
\url{https://arxiv.org/abs/2509.17886}
\end{botherref}
\endbibitem

\bibitem[\protect\citeauthoryear{Ghosh and Sathe}{1987}]{ghosh1987convergence}
\begin{botherref}
\oauthor{\bsnm{Ghosh}, \binits{J.}},
\oauthor{\bsnm{Sathe}, \binits{Y.}}:
Convergence of Bhattacharya bounds-revisited.
Sankhy{\=a}: The Indian Journal of Statistics, Series A,
37--42
(1987)
\end{botherref}
\endbibitem

\bibitem[\protect\citeauthoryear{Amari and Nagaoka}{2000}]{amari}
\begin{bbook}
\bauthor{\bsnm{Amari}, \binits{S.}},
\bauthor{\bsnm{Nagaoka}, \binits{H.}}:
\bbtitle{Methods of Information Geometry}.
\bpublisher{American Mathematical Society and Oxford University Press},
(\byear{2000})
\end{bbook}
\endbibitem

\bibitem[\protect\citeauthoryear{Takeuchi and
  Akahira}{2003}]{takeuchi2003second}
\begin{bchapter}
\bauthor{\bsnm{Takeuchi}, \binits{K.}},
\bauthor{\bsnm{Akahira}, \binits{M.}}:
\bctitle{Second order asymptotic efficiency in terms of asymptotic variances of
  the sequential maximum likelihood estimation procedures}.
In: \bbtitle{Joint Statistical Papers Of Akahira And Takeuchi},
pp. \bfpage{319}--\blpage{324}.
\bpublisher{World Scientific}, 
(\byear{2003})
\end{bchapter}
\endbibitem

\bibitem[\protect\citeauthoryear{Okamoto et~al.}{1991}]{okamoto1991asymptotic}
\begin{barticle}
\bauthor{\bsnm{Okamoto}, \binits{I.}},
\bauthor{\bsnm{Amari}, \binits{S.-I.}},
\bauthor{\bsnm{Takeuchi}, \binits{K.}}:
\batitle{Asymptotic theory of sequential estimation: Differential geometrical
  approach}.
\bjtitle{The Annals of Statistics}
\bvolume{19}(\bissue{2}),
\bfpage{961}--\blpage{981}
(\byear{1991})
\end{barticle}
\endbibitem

\bibitem[\protect\citeauthoryear{Smith}{2005}]{smith_intrinsic_crb}
\begin{barticle}
\bauthor{\bsnm{Smith}, \binits{S.T.}}:
\batitle{Covariance, subspace, and intrinsic Cramér–Rao bounds}.
\bjtitle{IEEE Trans. Signal Process.}
\bvolume{53}(\bissue{5}),
\bfpage{1610}--\blpage{1630}
(\byear{2005})
\end{barticle}
\endbibitem

\bibitem[\protect\citeauthoryear{Barrau and
  Bonnabel}{2013}]{boumal2013_intrinsic_crb}
\begin{bchapter}
\bauthor{\bsnm{Barrau}, \binits{A.}},
\bauthor{\bsnm{Bonnabel}, \binits{S.}}:
\bctitle{A note on the intrinsic Cramér–Rao bound}.
In: \beditor{\bsnm{Nielsen}, \binits{F.}},
\beditor{\bsnm{Barbaresco}, \binits{F.}} (eds.)
\bbtitle{Geometric Science of Information},
pp. \bfpage{377}--\blpage{386}.
\bpublisher{Springer},
\blocation{Berlin, Heidelberg}
(\byear{2013})
\end{bchapter}
\endbibitem

\bibitem[\protect\citeauthoryear{Bouchard et~al.}{2024}]{bouchard2024}
\begin{barticle}
\bauthor{\bsnm{Bouchard}, \binits{F.}},
\bauthor{\bsnm{Renaux}, \binits{A.}},
\bauthor{\bsnm{Ginolhac}, \binits{G.}},
\bauthor{\bsnm{Breloy}, \binits{A.}}:
\batitle{Intrinsic Bayesian Cramér–Rao bound with an application to covariance
  matrix estimation}.
\bjtitle{IEEE Transactions on Information Theory}
\bvolume{70}(\bissue{12}),
\bfpage{9261}--\blpage{9276}
(\byear{2024})
\doiurl{10.1109/TIT.2024.3480280}
\end{barticle}
\endbibitem

\bibitem[\protect\citeauthoryear{Mishra and Kumar}{2020}]{mishra2020}
\begin{bchapter}
\bauthor{\bsnm{Mishra}, \binits{K.V.}},
\bauthor{\bsnm{Kumar}, \binits{M.A.}}:
\bctitle{Generalized Bayesian Cramér–Rao inequality via information geometry of
  relative $\alpha$-entropy}.
In: \bbtitle{2020 54th Annual Conference on Information Sciences and Systems
  (CISS)},
pp. \bfpage{1}--\blpage{6}
(\byear{2020}).
\bcomment{IEEE}
\end{bchapter}
\endbibitem

\bibitem[\protect\citeauthoryear{Ashok~Kumar and
  Vijay~Mishra}{2020}]{ashok2020}
\begin{barticle}
\bauthor{\bsnm{Ashok~Kumar}, \binits{M.}},
\bauthor{\bsnm{Vijay~Mishra}, \binits{K.}}:
\batitle{Cram{\'e}r--Rao lower bounds arising from generalized Csisz{\'a}r
  divergences}.
\bjtitle{Information Geometry}
\bvolume{3}(\bissue{1}),
\bfpage{33}--\blpage{59}
(\byear{2020})
\end{barticle}
\endbibitem

\bibitem[\protect\citeauthoryear{Mishra and Kumar}{2021}]{mishra2021}
\begin{barticle}
\bauthor{\bsnm{Mishra}, \binits{K.V.}},
\bauthor{\bsnm{Kumar}, \binits{M.A.}}:
\batitle{Information geometry and classical Cramér–Rao-type inequalities}.
\bjtitle{Information Geometry}
\bvolume{45},
\bfpage{79}
(\byear{2021})
\end{barticle}
\endbibitem

\bibitem[\protect\citeauthoryear{Li and Zhao}{2023}]{Li}
\begin{barticle}
\bauthor{\bsnm{Li}, \binits{W.}},
\bauthor{\bsnm{Zhao}, \binits{J.}}:
\batitle{Wasserstein information matrix}.
\bjtitle{Information Geometry}
\bvolume{6}(\bissue{1}),
\bfpage{203}--\blpage{255}
(\byear{2023})
\end{barticle}
\endbibitem

\bibitem[\protect\citeauthoryear{Nishimori and Matsuda}{2025}]{nishimori2025}
\begin{botherref}
\oauthor{\bsnm{Nishimori}, \binits{H.}},
\oauthor{\bsnm{Matsuda}, \binits{T.}}:
On the attainment of the Wasserstein--Cramér–Rao lower bound: H. Nishimori et
  al.
Information Geometry,
1--12
(2025)
\end{botherref}
\endbibitem

\bibitem[\protect\citeauthoryear{Trillos et~al.}{2025}]{garciatrillos_wcrb}
\begin{botherref}
\oauthor{\bsnm{Trillos}, \binits{N.G.}},
\oauthor{\bsnm{Jaffe}, \binits{A.Q.}},
\oauthor{\bsnm{Sen}, \binits{B.}}:
Wasserstein--Cramér–Rao theory of unbiased estimation.
arXiv preprint arXiv:2511.07414
(2025)
\end{botherref}
\endbibitem

\bibitem[\protect\citeauthoryear{Amari and
  Matsuda}{2024}]{amari_matsuda_Wasserstein}
\begin{barticle}
\bauthor{\bsnm{Amari}, \binits{S.-i.}},
\bauthor{\bsnm{Matsuda}, \binits{T.}}:
\batitle{Information geometry of Wasserstein statistics on shapes and affine
  deformations}.
\bjtitle{Information Geometry}
\bvolume{7}(\bissue{2}),
\bfpage{285}--\blpage{309}
(\byear{2024})
\end{barticle}
\endbibitem

\bibitem[\protect\citeauthoryear{Fukushi et~al.}{2025}]{fukushi2024flatness}
\begin{botherref}
\oauthor{\bsnm{Fukushi}, \binits{A.}},
\oauthor{\bsnm{Nakanishi-Ohno}, \binits{Y.}},
\oauthor{\bsnm{Matsuda}, \binits{T.}}:
Flatness of location-scale-shape models under the Wasserstein metric.
arXiv preprint arXiv:2511.09959
(2025)
\end{botherref}
\endbibitem

\bibitem[\protect\citeauthoryear{Zuo}{2025}]{zuo2025}
\begin{botherref}
\oauthor{\bsnm{Zuo}, \binits{X.}}:
Advancing computations for Wasserstein gradient flows.
PhD thesis,
University of California, Los Angeles
(2025)
\end{botherref}
\endbibitem

\bibitem[\protect\citeauthoryear{Efron}{1975}]{efron1975}
\begin{botherref}
\oauthor{\bsnm{Efron}, \binits{B.}}:
Defining the curvature of a statistical problem (with applications to second
  order efficiency).
The Annals of Statistics,
1189--1242
(1975)
\end{botherref}
\endbibitem

\end{thebibliography}
\end{document}